\documentclass[11pt,a4paper]{amsart}

\usepackage{color}
\usepackage{amsmath,enumerate}
\usepackage{amsfonts}
\usepackage{amsthm}
\usepackage{array,calc,youngtab}
\usepackage{mathrsfs}
\usepackage{amssymb}
\usepackage[english]{babel}
\usepackage{graphicx}
\usepackage[all]{xy}
\usepackage{mathtools}
\usepackage[active]{srcltx}
\usepackage[parfill]{parskip}
\usepackage{verbatim}

\allowdisplaybreaks

\newtheorem{theorem}{Theorem}[section]
\newtheorem{lemma}[theorem]{Lemma}
\newtheorem{definition}[theorem]{Definition}
\newtheorem{corollary}[theorem]{Corollary}
\newtheorem{proposition}[theorem]{Proposition}
\newtheorem{remark}[theorem]{Remark}

\newtheorem{scholium}[theorem]{Scholium}

\newcommand{\Hom}{{\mathrm{Hom}}}

\newcommand{\eins}{\leavevmode\hbox{\small1\kern-3.8pt\normalsize1}}

\newcommand{\minus}{\scalebox{0.9}{{\rm -}}}
\newcommand{\plus}{\scalebox{0.6}{{\rm+}}}

\newcommand{\mR}{\mathbb{R}}
\newcommand{\mC}{\mathbb{C}}
\newcommand{\mN}{\mathbb{N}}
\newcommand{\mE}{\mathbb{E}}
\newcommand{\mZ}{\mathbb{Z}}

\newcommand{\mQ}{\mathbb{Q}}

\newcommand{\mK}{\mathbb{K}}
\newcommand{\mA}{\mathbb{A}}

\newcommand{\mg}{\mathfrak{g}}

\newcommand{\mk}{\mathfrak{k}}

\newcommand{\Span}{{\mathrm{Span}}}

\newcommand{\fg}{{\mathfrak g}}
\newcommand{\fh}{{\mathfrak h}}

\newcommand{\fk}{{\mathfrak k}}

\newcommand{\cD}{\mathcal{D}}

\newcommand{\cP}{\mathcal{P}}

\newcommand{\cC}{\mathcal{C}}

\newcommand{\cO}{\mathcal{O}}

\newcommand{\cB}{\mathcal{B}}
\newcommand{\cZ}{\mathcal{Z}}

\newcommand{\oa}{\bar{0}}
\newcommand{\ob}{\bar{1}}

\newcommand{\End}{{\rm End}}
\newcommand{\Tr}{{\rm Tr}}

\newcommand{\id}{{\rm id}}
\newcommand{\ad}{{\rm ad}}
\newcommand{\epa}{{\epsilon_1}}
\newcommand{\epb}{{\epsilon_2}}
\newcommand{\epc}{{\epsilon_3}}
\newcommand{\dea}{{\delta_1}}
\newcommand{\deb}{{\delta_2}}
\newcommand{\dec}{{\delta_3}}

\newcommand{\wA}{\mathscr{A}}
\newcommand{\wD}{\mathscr{D}}

\DeclarePairedDelimiter\abs{\lvert}{\rvert}%
\DeclarePairedDelimiter\norm{\lVert}{\rVert}%
\makeatletter
\let\oldabs\abs
\def\abs{\@ifstar{\oldabs}{\oldabs*}}
\let\oldnorm\norm
\def\norm{\@ifstar{\oldnorm}{\oldnorm*}}
\makeatother

\begin{document}

\title[Polynomial realisations of Lie superalgebras.]{Polynomial realisations of Lie (super)algebras and Bessel operators}

\author{Sigiswald Barbier and Kevin Coulembier}
\date{}

\begin{abstract}
We study realisations of Lie (super)algebras in Weyl (super)algebras and connections with minimal representations. The main result
is the construction of small realisations of Lie superalgebras, which we apply for two distinct purposes. Firstly it naturally introduces, and generalises,
the Bessel operators for Jordan algebras in the study of minimal representations of simple Lie groups. Secondly, we work out the theoretical realisation
concretely for the exceptional Lie superalgebra $D(2,1;\alpha)$, giving a useful hands-on realisation.
\end{abstract}

\maketitle


\section{Introduction}

Consider a complex simple Lie algebra $\fg$. In \cite{JosephReal}, Joseph determined the minimal natural number $n=n_{\wD}(\fg)$ for which $\fg$ can be embedded in $\mathscr{D}_n$, where~$\mathscr{D}_n$ is a canonically defined left quotient field of $\wA_n$, the Weyl algebra of differential operators on $\mR^n$ with complex polynomial coefficients. This embedding extends to a morphism from the universal enveloping algebra $U(\fg)$ to $\mathscr{D}_n$ which has kernel $J_0$, known as the {\em Joseph ideal}.

According to Definition 4.6 of Gan and Savin in~\cite{GS}, a unitary representation of a simple real Lie group $G$ is called {\em minimal} if the annihilator ideal of the derived representation of the universal enveloping algebra of Lie$(G)_{\mC}$ is the Joseph ideal. In particular, minimal representations attain the minimal Gelfand-Kirillov dimension amongst all unitary representations of~$G$.

In \cite{Dib}, Dib introduced a second order differential operator on Jordan algebras, called the {\em Bessel operator}, yielding a system of differential equations generalising the Bessel differential equation. In \cite{BesselOperators}, Hilgert, Kobayashi and M\"ollers obtained a unifying construction of ``small" unitary representations, in particular minimal ones, of a large class of real simple Lie groups~$G$, by using the Bessel operator on the Jordan algebra $V$, linked to $G$ by the Tits-Kantor-Koecher (TKK) construction, as one of the main tools. This {\rm TKK} construction provides a correspondence between certain Lie algebras equipped with three-term $\mZ$-gradation and Jordan algebras. Under this correspondence ${\rm Lie}(G)$ is equal to the ``conformal algebra'' $\mathfrak{co}(V)$. The results in \cite{BesselOperators} fit into a large project, studying various properties of minimal representations using Bessel operators, on which remarkable progress has been made during the last decade. Bessel operators, also referred to as fundamental differential operators, appeared earlier in the study of specific examples of minimal representations in e.g.~\cite{Ko11, Opq, KM1, KM2, Sahi}. Some recent further results on Bessel operators can be found in \cite{HKMO, Mollers2, KoSurvey}.

The purpose of the current paper is threefold:
\begin{enumerate}
\item Finding a natural constructive way to introduce Bessel operators in the study of minimal representations of Lie groups.
\item Starting the systematic study of the analogue of the Joseph ideal and minimal representations for Lie {\em superalgebras}, extending the specific cases already considered in~\cite{CSS}.
\item Finding compact explicit realisations of the exceptional Lie superalgebras $D(2,1;\alpha)$, $F(4)$ and $G(3)$.
\end{enumerate}

To achieve part (1) we will work out explicitly certain realisations of $\fg$, for any given three-term $\mZ$-gradation, in a Weyl algebra. The existence of this realisation, guaranteed by general arguments in~\cite{Conze}, was used in~\cite{JosephReal} to obtain an upper bound on $n_{\wD}(\fg)$. We will prove that this realisation is a generalisation of the representation of $\mathfrak{co}(V)$ in~\cite{BesselOperators} to the setting of Jordan {\em pairs}, from the specific case of {\em simple unital} Jordan {\em algebras}. This construction makes a new direct link between \cite{JosephReal} and \cite{BesselOperators}, which helps to further explain, from a different perspective, {\em why} the Bessel operators are so useful in the construction of minimal representations. At the same time, it now follows by construction that the Bessel operators lead to a representation of $\mathfrak{co}(V)$. It is an interesting open question in which generality the techniques in~\cite{BesselOperators, KM2} can be extended from simple unital Jordan algebras to Jordan pairs.

Because of the goal in part (2), we carry out the construction of (1) immediately for the case where~$\fg$ is a Lie superalgebra. For this, we need to generalise some technical results of \cite{Berezin, Conze}, concerning universal enveloping algebras, to the case of ($\mZ$-graded) superalgebras. In particular, we obtain a construction of small polynomial realisations for 3-graded Lie superalgebras, which will be the starting point of the study of an interesting class of representations of Lie supergroups in subsequent work, following the spirit of \cite{BesselOperators}. 

The small polynomial realisations of Lie superalgebras have yet another application, as mentioned in aim (3). The exceptional simple basic classical Lie superalgebras $D(2,1;\alpha)$, $F(4)$ and $G(3)$, see e.g. \cite{Musson}, do not admit small dimensional representations, like the ones for the families $\mathfrak{osp}$ and $\mathfrak{sl}$. Therefore, there are no convenient matrix realisations available. Concretely, for the one parameter family $D(2,1;\alpha)$ of deformations of $\mathfrak{osp}(4|2)$, the smallest representation (the adjoint representation) of $D(2,1;\alpha)$ is 17-dimensional as soon as $\alpha\not\in\mQ$, see \cite{Joris}. This is in sharp contrast with the undeformed superalgebra $\mathfrak{osp}(4|2)$, which has a 6-dimensional representation, {\it viz.} the natural representation. So for generic $D(2,1;\alpha)$, contrary to $\mathfrak{osp}(4|2)$, the smallest matrix realisation is not convenient.

We apply our results to derive which convenient polynomial realisations exist of the exceptional Lie superalgebras and work them out very explicitly for the case $D(2,1;\alpha)$. For every fixed parameter $\alpha$, this yields a one parameter family of realisations, as polynomial differential operators on $2|2$-dimensional superspace. We also determine when the corresponding representation on polynomials is irreducible, revealing information on the expected structure of orbits on which representations can be constructed using the methods of \cite{BesselOperators} and a candidate for the minimal representation.

The paper is organised as follows. In Section \ref{secprel} we gather necessary preliminaries on superalgebras and Jordan (super)pairs. In Section \ref{Constr} we study the universal enveloping algebra of Lie superalgebras and use this to construct useful embeddings of Lie superalgebras in (completions of) Weyl superalgebras. In Section \ref{secBessel} we use a specific example of the aforementioned realisations, in the case of a 3-graded Lie superalgebra, to define the Bessel operators for the associated Jordan superpair. We show that this generalises the known Bessel operators for unital Jordan algebras. In Section \ref{secD} we use the results of Section \ref{Constr} to construct a compact explicit realisation of the Lie superalgebras $D(2,1;\alpha)$. In Appendix \ref{technical} we carry out some technical calculations concerning enveloping algebras, which are essential for the construction in Section \ref{Constr}.


\section{Preliminaries}\label{secprel}

We refer to \cite{CW, Musson} for a general introduction to (Lie) superalgebras and to \cite{DM} for an introduction to supermanifolds. We will always work over the field $\mK$ which is either $\mR$ or $\mC$. We set $\mZ_2=\mZ/2\mZ=\{\oa,\ob\}$ and set $|a|=\oa$, respectively $|a|=\ob$, for $a$ an even, respectively odd, element of some $\mK$-super vector space. Unless specified otherwise, Lie superalgebras and Jordan superalgebras are assumed to be finite dimensional. In our convention $0$ is a natural number, so $\mN=\{0,1,2,\cdots\}$.

\subsection{Lie superalgebras and symmetrisation}\label{PrelSec1}
Let $A$ be an associative $\mK$-superalgebra and let $S(A)$ be the supersymmetric algebra of the super vector space $A$. To distinguish between multiplication in $A$ and $S(A)$, we denote the product of two elements $\alpha,\beta\in S(A)$ by $\alpha \bullet \beta=(-1)^{\abs{\alpha}\abs{\beta}}\beta\bullet\alpha$. So, for $\alpha,\beta\in A$, we always have $\alpha\bullet\beta\not=\alpha\beta$, as both sides live in different spaces.

Consider the symmetrisation map $\sigma$  from $S(A)$ to $A$. On elements of the form $a_1\bullet\cdots \bullet a_{p+q}$ with homogeneous $a_i\in A$, ordered such that $a_i$ is even for $i\leq p$ and odd for $i>p$, we have
\[
\sigma(a_1\bullet\cdots\bullet a_p\bullet a_{p+1} \bullet \cdots \bullet a_{p+q}) = \frac{1}{(p+q)!}\sum_{\tau \in S_{p+q}}(-1)^{\abs{\tau}} a_{\tau(1)} \cdots a_{\tau(p+q)},
\]
where~$S_{p+q}$ is the permutation group of $p+q$ objects and 
\begin{align} \label{sign permutation}
 \abs{\tau } = \sum_{i=p+1}^{p+q-1} \sum_{j=i+1}^{p+q} \left[ \tau^{\minus 1}(i) > \tau^{\minus 1}(j) \right].
\end{align}
Here we used the Iverson bracket
\begin{align*}
[ \text{expression} ] = \begin{cases}
0 \qquad \text{if the expression is false} \\
1 \qquad \text{if the expression is true}.
\end{cases}
\end{align*}

For a $\mK$-Lie superalgebra $\fg$ we denote the universal enveloping algebra by $U(\mg)$. The restriction of $\sigma$ from $S(U(\fg))$ to $S(\fg)$ yields a vector space isomoprhism
\begin{equation}\label{sigmaisom}
\sigma\, \colon \;S(\mg)\; \tilde\to \;U(\mg).
\end{equation}
In the multiplication on $S(\fg)$ we will leave out $\bullet$ as there is no ambiguity.

A $\mZ$-grading of a Lie superalgebra $\fg$ is a decomposition as super vector spaces
\begin{equation}\label{Zgrad}
\mg=\bigoplus_{i \in \mZ} \mg_{i} ,
\end{equation}
where~$[\mg_{i},\mg_{j}]\subset \mg_{i+j}.$  For such a grading we put $\mg_{\plus} = \bigoplus_{i >0} \mg_i,$  $ \mg_{\minus}=\bigoplus_{i <0} \mg_{i}$ and $\mk=\mg_0 \oplus \mg_{\plus}$. When $\fg_{\plus}=\fg_1$ and $\fg_{\minus}=\fg_{\minus 1}$, we say that $\fg$ is $3$-graded. The grading in \eqref{Zgrad} is not to be confused with the $\mZ_2$-grading $\fg=\fg_{\oa}\oplus\fg_{\ob}$.

 By a character of a $\mK$-Lie superalgebra we mean an (even) Lie superalgebra morphism to $\mK$.  Note that we can extend any character $\lambda$ of $\fg_0$, in the notation of \eqref{Zgrad}, to a character $\lambda:\fk\to \mK$ by setting $\lambda(\fg_{\plus})=0$. We will often silently make this identification. Moreover, for a character $\lambda$ of $\fg_0$ we consider the 1-dimensional $\fk$-module $\mK_\lambda$, where~$Xv=\lambda(X)v$ for all $X\in \fk$ and $v\in \mK_\lambda$.

\subsection{Some conventions}\label{somconv}
We will use~$(m,n)$-multiple indices, which we define as
$$
\mN^{m|n}= \mN^m \times \{0,1\}^n, \quad
K=(k_1,k_2,\ldots, k_m | k_{m+1},\ldots, k_{m+n}) \in \mN^{m|n} $$
and for which we introduce the notation 
$$\abs{K}= \sum_{i=1}^{m+n} k_i\;\;\quad\mbox{and} \qquad
K!=k_1! \cdots k_m!.$$
For $K,L\in \mN^{m|n}$, we say that $L < K$ if $l_i \leq k_i$ for $i=1,\ldots, m+n$ and $L\not= K$.

The Bernoulli numbers $B_i$ are recursively defined by \cite{Weisstein} \begin{align}\label{Bernoulli numbers}
B_i = -\sum_{j=0}^{i-1} \binom{i}{j} \frac{B_j}{i-j+1}\mbox{ for }i>1,\; B_0=1, \mbox{ and }\; B_1=-1/2.
\end{align} 

We will use the convention 
$$\ad_Z(Y):=[Z,Y]\quad\mbox{and}\quad \widetilde{\ad}_Z(Y):=[Y,Z],$$
where~$[\cdot,\cdot]$ denotes the Lie superbracket in case~$Y$ and $Z$ are elements of a Lie superalgebra, or the supercommutator in case~$Y$ and $Z$ are elements of an associative superalgebra.


For an $\mR$-super vector space $U$, we introduce the affine superspace 
$$\mA(U):=(U_{\oa},\cC^\infty_{U_{\oa}}\otimes \Lambda U_{\ob}^\ast),$$
which is a real supermanifold. In particular we set $\mA^{n|m}:=\mA(\mR^{n|m})$.

A super derivation $D$ on a $\mK$-superalgebra $A$ is some $D\in \End_{\mK}(A)$ satisfying the super Leibniz rule. Concretely, the super derivations are the linear combinations of homogenous $D\in \End_{\mK}(A)$ satisfying
$$D(ab)=D(a)b+(-1)^{|D||a|}a D(b)$$
for all $a,b\in A$, with $a$ homogeneous. Let $V$ be a finite dimensional super vector space and $X_1,\ldots, X_{m+n}$ a basis for $V,$ where~$X_i$ is even for $i\leq m$ and odd for $i>m$. We define the partial derivatives $\partial^i$, as the unique derivation on the superalgebra $S(V)$ satisfying $\partial^i X_j= \delta_{ij}$.
They satisfy the relations $ \partial^i\partial^j=(-1)^{\abs{X_i}\abs{X_j}} \partial^j\partial^i$ and hence canonically generate $S(V^\ast)$.
We define a basis of $S(V^\ast)$ by
$$\partial^K  = (\partial^1)^{k_1} \cdots (\partial^{m+n})^{k_{m+n}} \quad\mbox{ with }\;\quad K\in \mN^{m|n}.$$

On a supermanifold $M$, the sheaf of functions, respectively of differential operators, is denoted by~$\cO_M$, respectively $\cD_M$. The global sections are denoted by $\Gamma(\cO_M)$ and $\Gamma(\cD_M)$. In case~$M=\mA(U)$, for a real super vector space $U$, and for a homogeneous basis $X_j\in U^\ast$, the partial derivatives $\partial^j$ defined in the above paragraph are elements of $\Gamma(\cD_{\mA(U)})$. Together with the elements of $\cO_{\mA(U)}$ they generate~$\cD_{\mA(U)}$.

\subsection{Polynomial realisations}

Consider a finite dimensional $\mK$-super vector space $V$. We define the (super) Weyl algebra, also known as the Weyl-Clifford algebra, $\wA(V)$ as the $\mK$-subalgebra of $\End_\mK(S(V))$ generated by multiplication with elements of $V$ and the super derivations on the algebra $S(V)$. In particular we have a natural identification of super vector spaces
\begin{equation}\label{ASS}\wA(V)\;\cong\; S(V)\otimes S(V^\ast)\;\,\subset\, \End(S(V)),\end{equation}
where~$V^\ast$ is interpreted as the space spanned by the partial derivatives. When we take $V=\mK^{n|m}$ we denote this by $\wA_{n|m}(\mK)=\wA(\mK^{n|m})$. We will consider $\wA(V)$ both as an associative algebra and as an infinite dimensional Lie superalgebra with bracket given by the super commutator. Note that we have a canonical embedding of $\wA(V)$ into $\Gamma(\cD_{\mA(V^\ast)})$ for $V$ real.

We define a polynomial realisation of a $\mK$-Lie superalgebra $\fg$ to be an injective Lie superalgebra morphism 
$\phi:\fg\hookrightarrow \wA_{n|m}(\mK)$ for some $n,m\in\mN$. Note that if $\fg$ admits a faithful representation on a finite dimensional vector space $V$ then there is automatically a realisation in $\wA(V)$, contained in $V\otimes V^\ast$ under \eqref{ASS}. This is referred to as a matrix realisation. 

Clearly the canonical representation of $\wA(V)$ on $S(V)$ is faithful. This implies that for any associative algebra $A$ with algebra morphism $\phi:A\to \wA(V)$, the annihilator ideal in $A$ of the induced representation on $S(V)$ is given by the kernel of $\phi$.

We consider the usual topology on $S(V)$, where open neighbourhoods are given by powers of the maximal ideal cancelling $0$. We define $\widehat{S}(V^\ast)\subset \End(S(V))$ as the completion of $S(V)$. Finally we denote the subalgebra of $\End(S(V))$ generated by $\wA(V)$ and $\widehat{S}(V^\ast)$ by $\widehat{\wA}(V)$.



\subsection{Jordan superpairs and superalgebras}

\subsubsection{Jordan superpairs}\label{subsecJSP}
A $\mK$-Jordan superpair $V$ is a pair of $\mK$-super vector spaces $(V_{\plus}, V_{\minus})$ equipped with two even $\mK$-trilinear products, known as the Jordan triple products,  \[ 
\{ \cdot, \cdot, \cdot \}^{\sigma} \colon V_{\sigma} \times V_{-\sigma} \times V_{\sigma} \to V_{\sigma},
\]
$\sigma \in \{+,-\}$, which satisfy symmetry in the outer variables
\begin{align*}
\{x,y,z\}^\sigma &= (-1)^{\abs{x}\abs{y}+\abs{y}\abs{z} + \abs{z}\abs{x}} \{z,y,x\}^\sigma,
\end{align*}
and the  $5$-linear identity
\begin{align*}
\{ x,y, \{ u,v,w \}^\sigma \}^\sigma  &- \{ \{ x,y,u\}^\sigma ,v,w\}^\sigma \\&=(-1)^{(\abs{x}+\abs{y})(\abs{u}+\abs{v})} (-\{u	,\{v,x,y\}^\sigma,w\}^\sigma + \{u,v,\{x,y,w\}^\sigma  \}^\sigma),
\end{align*}
for homogeneous $x,z,u,w \in V_\sigma$ and $y,v \in V_{-\sigma}$.

Define the following operators
\[
D^\sigma \colon V_\sigma \times V_{-\sigma} \to \End (V_\sigma); \qquad (x,y) \mapsto D_{x,y}^\sigma,
\]
where~$D^{\sigma}_{x,y}(z) = \{x,y,z\}^\sigma.$ We can rewrite the $5$-linear identity with these operators as 
\begin{align} \label{5 linear operator identity}
[D^\sigma_{x,y},D^\sigma_{u,v}]&=D^\sigma_{\{x,y,u\},v} -(-1)^{(\abs{x}+\abs{y})(\abs{u}+\abs{v})}D^\sigma_{u,\{v,x,y\}} \\
&=D^\sigma_{x,\{y,u,v\}} -(-1)^{(\abs{x}+\abs{y})(\abs{u}+\abs{v})} D^\sigma_{\{u,v,x\},y}. \nonumber
\end{align}
We also introduce the operators 
\begin{align}\nonumber
&P^\sigma \colon V_\sigma \times V_{\sigma} \to {\Hom}(V_{-\sigma},V_\sigma); \qquad (x,y) \mapsto P_{x,y}^\sigma,\\\nonumber
&P^\sigma_{x,y} (z) := (-1)^{\abs{y}\abs{z}} \frac{1}{2}D^\sigma_{x,z}(y)\qquad\mbox{for }x,y\in V_\sigma \mbox{ and }z\in V_{-\sigma}.
\end{align}
In the following we will omit $\sigma$ from the notation, as the upper index of $D_{x,y}$ and $P_{x,y}$ is determined by $x$ and $y$, and similarly for $\{\cdot,\cdot,\cdot\}$.

The simple finite dimensional Jordan superpairs over an algebraic closed field of zero characteristic were classified in \cite{Krutelevich}.

\subsubsection{The Tits-Kantor-Koecher construction}
For a Jordan superpair we define the structure algebra $\mathfrak{str}(V)$ as the algebra spanned by $D_{x,u}$ for $x \in V_{\plus}$ and $u \in V_{\minus}$ and with the supercommutator as bracket. From \eqref{5 linear operator identity} it follows that $\mathfrak{str}(V)$ is a well-defined subalgebra of $\mathfrak{gl}(V_{\plus})$.

We can associate a 3-graded Lie superalgebra ${\rm TKK}(V)$ to the Jordan superpair $V$ in the following way, see section 2 in \cite{Krutelevich}.
As vector spaces we have \begin{equation}\label{TKKV}{\rm TKK}(V)= V_{\plus} \oplus \mathfrak{str}(V) \oplus V_{\minus}.\end{equation}
The Lie bracket on ${\rm TKK}(V)$ is defined as follows:
\begin{align*}
[x,u]&=D_{x,u} \\
[D_{x,u},y]&=\{x,u,y\}=D_{x,u}(y) \\
[u,D_{x,v}]&=\{ u,x,v\}=-(-1)^{(\abs{x}+\abs{v})\abs{u}}D^*_{x,v}(u) \\
[D_{x,u},D_{y,v}]&=D_{\{x,u,y\},v} - (-1)^{(\abs{x}+\abs{u})(\abs{y}+\abs{v})} D_{y,\{v,x,u\}} \\
[x,y]&=[u,v]=0,
\end{align*}
for $x,y \in V^+$, $u,v \in V^-$ and extended linearly and anti-commutatively.
Here we introduced $D_{x,v}^\ast= -(-1)^{\abs{x}\abs{v}} D_{v,x}$.

Conversely, with each finite dimensional 3-graded Lie superalgebra $\mg=\mg_{\plus 1} \oplus \mg_0 \oplus \mg_{\minus 1}$
 we can associate a  Jordan superpair by $\mathcal{J}(\mg)= (\mg_{\plus 1},\mg_{\minus 1})$ with Jordan triple product defined as
\[
\{x^\sigma,y^{-\sigma},z^\sigma \} = [[x^\sigma,y^{-\sigma}],z^\sigma].
\]

\begin{definition}
A 3-graded Lie superalgebra $\mg=\mg_{\plus 1} \oplus \mg_0 \oplus \mg_{\minus 1}$ is called \textbf{Jordan graded} if 
\[
[\mg_{\plus 1},\mg_{\minus 1}]=\mg_0 \text{ and } \mg_0 \cap Z(\mg)=0.
\]
\end{definition}
We have the following, by Lemmata 4, 5 and 6 in~\cite{Krutelevich}.
\begin{proposition}\label{TKK}
Let $\mg$ be a finite dimensional Jordan graded Lie superalgebra, then
\begin{itemize}
\item ${\rm TKK}(\mathcal{J}(\mg))$ is isomorphic to $\mg$.
\item $\mathcal{J}(\mg)$ is simple if and only if $\mg$ is simple.
\end{itemize}
Let $V$ be any finite dimensional Jordan superpair, then $\mathcal{J} ({\rm TKK}(V))$ is isomorphic to $V$.
\end{proposition}

\subsubsection{Jordan superalgebras }\label{secJalg}
One can interpret (unital) Jordan superalgebras as a special class of Jordan superpairs.
A unital Jordan $\mK$-superalgebra is a $\mK$-superalgebra $J$ with a $\mK$-bilinear product $\ast$ which satisfies
\begin{itemize}
\item $x \ast y = (-1)^{\abs{x}\abs{y}} y\ast x$ \qquad (commmutativity)
\item $(-1)^{\abs{x}\abs{z}}[L_x,L_{y\ast z}]+(-1)^{\abs{y}\abs{x}}[L_y,L_{z\ast x}]+(-1)^{\abs{z}\abs{y}}[L_z,L_{x\ast y}]=0$ \qquad (Jordan identity).
\end{itemize}
and  containing a unit element $e$. Here $L_x\in \End(J)$ is (left) multiplication with $x\in J$.

Define the following triple product for a unital Jordan algebra $J$: \[ \{ x,y,z \}=2\left(x \ast (y\ast z) +(x \ast y) \ast z -(-1)^{\abs{x}\abs{y}}y \ast (x \ast z) \right). \]

This product satisfies the conditions of a Jordan triple product. 
Therefore $(J,J)$ with this triple product is a Jordan superpair.

The operators $D_{x,y}$ and $P_{x,y}$ for the superpair $(J,J)$, as introduced in Subsection \ref{subsecJSP}, can now be expressed using the left multiplication operator instead of the triple product. We get
\begin{align} \label{expression D for Jordan algebra}
D_{x,y}= 2 L_{x\ast y}+2[L_x,L_y] \qquad \text{ and } \qquad P_{x,y} =L_x L_y +(-1)^{\abs{x}\abs{y}}L_y L_x-L_{x\ast y}.
\end{align}

To establish which Jordan graded Lie superalgebras are coming from a unital Jordan superalgebra under the {\rm TKK}-construction, we need the concept of minimal pairs.
\begin{definition}
A \textbf{short subalgebra} of a Lie superalgebra $\mg$ is an $\mathfrak{sl}(2)$ subalgebra spanned by even elements $e,h,f$ such that 
\begin{itemize}
\item The eigenspace decomposition of $ad(h)$ defines a $3$-grading of $g$.
\item $[h,e]=-e$, $[h,f]=f$, $[e,f]=h$.
\end{itemize}
\end{definition}
For  ${\rm TKK}(J,J)$ and unit $e\in J$, a short subalgebra is given by $e:=(0,0,e), h:=(0,L_e,0)$ and $ f:=-(\frac{e}{2},0,0)$, using the identification in equation \eqref{TKKV}. We denote this subalgebra by $\mathfrak{a}_J$.
\begin{definition}
A $3$-graded Lie superalgebra $\mg=\mg_{\plus 1} \oplus \mg_0 \oplus \mg_{\minus 1}$ is called \textbf{minimal} if any non-trivial ideal $I$ of $\mg$ intersects $\mg_{\minus 1}$ non-trivially.
\end{definition}

We call a Lie algebra $\mg$ with a short subalgebra $\mathfrak{a}$, such that the corresponding $3$-grading is minimal, a minimal pair $(\mg, \mathfrak{a})$. With each minimal pair $(\mg, \mathfrak{a})$ we can associate a unital Jordan superalgebra $J_{\mg, \mathfrak{a}}:=\mg_{\minus 1}$ with product
\[
x\bullet y = [[f,x],y].
\]

\begin{proposition}[Theorem 5.15 in~\cite{Cantarini}]
The functor $\mathcal{F}\colon J \to ({\rm TKK}(J,J), \mathfrak{a}_J)$ is an equivalence of the categories of unital Jordan superalgebras (with epimorphisms as morphisms) to the category of minimal pairs (with as morphisms epimorphisms $\phi$ such that $\phi(\mathfrak{a})=\mathfrak{b}$ for the short subalgebras $\mathfrak{a}, \mathfrak{b}$). The inverse is given by $(\mg, \mathfrak{a}) \to J_{\mg, \mathfrak{a}}$.
\end{proposition}

\begin{remark}{\rm
\label{remarkprel}
In conclusion, a Jordan graded Lie superalgebra corresponds to a unital simple Jordan algebra if and only if it is a simple Jordan graded Lie superalgebra containing a short subalgebra.}
\end{remark}

We will use the short-hand notation $\mathfrak{co}(J)$ for ${\rm TKK}(J,J)$.

\section{Construction of small polynomial realisations for Lie superalgebras}\label{Constr}

Our main result of this section is summarised in the following theorem. 
\begin{theorem}\label{ThmGen}
Consider a finite dimensional $\mK$-Lie superalgebra $\fg$.
\begin{enumerate}[(i)]
\item There is a Lie superalgebra embedding $\fg \;\hookrightarrow\; \widehat{\wA}(\fg)$.
\item For any $\mZ$-grading of $\fg$, with $\fg=\fg_{\minus}\oplus\fk$ and any character $\lambda$ of $\fg_0$, there is a Lie superalgebra morphism 
$$\phi_\lambda:\,\fg\;\to\; \wA(\fg_{\minus}),$$
which is injective if and only if $U(\fg)\otimes_{U(\fk)}\mK_\lambda$ is a faithful $\fg$-module.
\end{enumerate}
\end{theorem}
In case~$\fg$ is a Lie algebra, Theorem \ref{ThmGen}(i) follows from Theorem 3 in~\cite{Berezin}, while Theorem \ref{ThmGen}(ii) can be obtained from (the proof of) Proposition 2.2 in~\cite{Conze}.

Furthermore, we will determine the explicit form of $\phi_\lambda$ in Theorem \ref{ThmGen}(ii) in case of a 3-grading, which is new also for Lie algebras. In Section \ref{secBessel}, this will lead to the natural appearance of Bessel operators, containing the ones in~\cite{Dib, FK, BesselOperators, Opq, Sahi} as a special case.

\subsection{Superversion of a result of Berezin}
In this section we generalise the approach of \cite{Berezin} to superalgebras, where we work out most of the technical details in Appendix \ref{technical}. We consider the left regular representation of a Lie superalgebra $\mg$ on $U(\mg)$ by left multiplication. Using the isomorphism \eqref{sigmaisom}, this yields a $\mg$-representation $\pi$ on $S(\mg).$ 

We prove that the image of $\pi:\fg\to\End(S(\fg))$ is actually contained in $\widehat{\wA}(\fg)$ and obtain an explicit expression in the following theorem. Therefore we choose a basis
$ X_i,$ $1\leq i \leq m+n$ of $\fg$, where~$X_i$ is even if $i\leq m$ and odd if $i > m$.

\begin{theorem}\label{Berezin theorem}
The action $\pi$ defined by $$\pi(X) Y := \sigma^{\minus 1}( X \sigma(Y)) \qquad \text{ for all } X \in \mg, \quad Y \in S(\mg)$$ is given by
\begin{align}  \label{supercase of Berezin}
\pi(X) Y  = \sum_{K\in \mN^{m|n}} \frac{(-1)^{\abs{K}}B_{\abs{K}}}{K!}   s^K_{\mg}(X)  \; \partial^K Y,
\end{align}
where the operator $s^K_{\mg}\in \End(\fg)$ is defined as
\begin{align} \label{definition sKg}
s^K_{\mg}(X):=\sigma\left(\widetilde{\ad}_{X_1}^{\bullet k_1} \bullet \cdots\bullet \widetilde{\ad}_{X_{m+n}}^{\bullet k_{m+n}}\right)X\qquad \text{ for all } X \in \mg,
\end{align}
and with $B_{\abs{K}}$ the Bernoulli numbers \eqref{Bernoulli numbers}.
\end{theorem}

First we observe that this theorem implies Theorem \ref{ThmGen}(i). Indeed, the expression in equation~\eqref{supercase of Berezin} confirms that $\pi(X)\in \widehat{\wA}(\fg)$ for any $X\in \fg$. Furthermore, the injectivity of $\pi:\fg\to  \widehat{\wA}(\fg)$ follows immediately from the fact that the left regular representation (and hence $\pi$) is faithful.

\begin{proof}[Proof of Theorem \ref{Berezin theorem}]
For $Y$ of degree one, i.e. $Y \in \mg$, the claim reduces to
\[
X\sigma(Y) = \sigma(XY)+\frac{1}{2} \sigma([X,Y]),
\]
which is clearly true. Now assume \eqref{supercase of Berezin} holds for all elements in $S(\mg)$ which have lower degree than $Y$.
We can rewrite Lemma~\ref{lemma 1} as
\begin{align*}
X \sigma(Y)= \sigma(XY)-\sum_{K>0} \frac{(-1)^{\abs{K}} }{(\abs{K}+1)K!} s^K_{\mg}(X) \sigma( \partial^K Y).
\end{align*}
Since the degree of $ \partial^K Y$ is now strictly lower than the degree of $Y$ we can apply our induction hypothesis on $s^K_{\mg}(X) \sigma( \partial^K Y)=\sigma(\pi(s_{\fg}^K(X))\partial^KY)$. This leads to
\begin{align*}
X\sigma(Y)= &\sigma(XY) -\sum_{K>0} \frac{(-1)^{\abs{K}} }{(\abs{K}+1)K!}  \sum_{L \in \mN^{m|n}} \frac{(-1)^{\abs{L}}B_{\abs{L}}}{L!}  \sigma\left(s^L_{\mg}\circ s^K_{\mg}(X) \; \partial^L \partial^K Y\right) \\
=& \sigma(XY) -\sum_{K>0} \sum_{L < K} \frac{(-1)^{\abs{K-L}+\abs{L}} B_{\abs{L}}}{(\abs{K-L}+1)(K-L)!L!}     \sigma\left(s^L_{\mg}\circ s^{K-L}_{\mg}(X) \; \partial^L \partial^{K-L} Y\right).
\end{align*}
From Lemma~\ref{Lemma 3}, we have
\begin{align*}
\sum_{L < K} \frac{(-1)^{\abs{K-L}+\abs{L}} B_{\abs{L}}}{(\abs{K-L}+1)(K-L)!L!}     s^L_{\mg}\circ s^{K-L}_{\mg}(X) \; \partial^L \partial^{K-L} Y 
= - \frac{(-1)^{\abs{K}}}{ K!} B_{\abs{K}} s^K_{\mg}(X) \partial^K (Y).
\end{align*}
Using this, we obtain
\begin{equation}
\label{eqBerezin}
X\sigma(Y) = 
\sigma(XY) +\sum_{K>0}\frac{(-1)^KB_{|K|}}{K!} \sigma(s^K_{\mg} (X )\partial^K Y),
\end{equation}
which proves the theorem.
\end{proof}

\begin{corollary} \label{Lemma rewrite of theorem Berezin}
We have
\[
\sigma(XY) = (-1)^{\abs{X}\abs{Y}} \sigma(Y) X + \sum_{K >0} \frac{(-1)^{\abs{K}}C_{\abs{K}}}{K!} \sigma(s^K_{\mg}(X) \; \partial^K Y),
\]
where~$B_{i} =- C_{i}$ for $i \geq 2$ and $C_{1}=B_{1}=-1/2$.
\end{corollary}
\begin{proof}
Follows immediately from Lemma~\ref{lemma Lie bracket of sigma} and equation \eqref{eqBerezin}.
\end{proof}

\subsection{A method to construct small polynomial realisations}
Let $\mg$ be a finite-dimensional $\mZ$-graded Lie superalgebra, where we maintain the notation of Subsection \ref{PrelSec1}. We consider a character $\lambda:\fg_0\to\mK$, interpreted as a character of $\fk$. For any such $\lambda$ we will use Theorem \ref{Berezin theorem} to construct a realisation of $\mg$ on $S(\mg_{\minus})$, by reinterpreting the parabolic Verma module of scalar type $U(\mg) \otimes_{U(\mathfrak{k})} \mK_\lambda$.

As vector spaces we have
\[
U(\mg) \otimes_{U(\mathfrak{k})} \mK_\lambda \cong U(\mg_{\minus}) \otimes \mK \cong S(\mg_{\minus}).
\]
The $\mg$-representation on $U(\mg) \otimes_{U(\mathfrak{k})} \mK_\lambda$ hence yields a $\fg$-representation $\pi$ on $S(\mg_{\minus})$ using  the symmetrisation map $\sigma$ and \[\mu\colon S(\mg_{\minus})\otimes \mK\; \tilde\to \;S(\mg_{\minus}) \; ; \; Y \otimes a \to aY.\]
We will prove that the image of $\pi$ is contained in $\wA(\fg_{\minus}).$ 

Let $(r|s)$ be the dimension of $\mg_{\minus}$ and choose a basis $X_i$ with $1\le i\le r$ of the even part and a basis $X_{j+ r}$ with $1\le j\le s$ of the odd part. To be able to give an explicit expression of this action we define the following elements of $\mg_{\minus}$ and $\mk$.
For $1\le i\le r+s$ and $X \in \mk$, we set \[
W_i(X) +H_i(X):=  [X,X_i] ,
\]
uniquely defined by the condition~$W_i(X)\in\mg_{\minus}$ and $H_i(X)\in\mk$.
For $K_1 \in \mN^{r|s} \backslash \{ 0 \}$, we put \[
W_{i,K_1}(X)+H_{i,K_1}(X) := \frac{(-1)^{\abs{K_1}}C_{\abs{K_1}}}{K_1!} s^{K_1}_{\mg_{\minus}}(H_i(X)), 
\]
with~$W_{i,K_1}(X)\in\mg_{\minus}$ and $H_{i,K_1}(X)\in\mk$. The operator $s_{\fg_{\minus}}^{K_1}$ is defined as in \eqref{definition sKg}, but now adjoining only the basis elements of $\mg_{\minus}$ instead of the whole $\mg$.
Recursively we also set, for $K_1,\ldots, K_j \in \mN^{r|s} \backslash \{ 0 \}$,
\[
W_{i,K_1,\ldots,K_j}(X)+H_{i,K_1,\ldots,K_j}(X) := \frac{(-1)^{\abs{K_j}}C_{\abs{K_j}}}{K_j!} s^{K_j}_{\mg_{\minus}}(H_{i,K_1,\ldots, K_{j-1}}(X)),
\]
where again $W_{i,K_1,\ldots,K_j}(X)\in\mg_{\minus}$ and $H_{i,K_1,\ldots,K_j}(X)\in\mk$.

\begin{theorem} \label{Polynomial realisation of a graded Lie algebra}
For any character $\lambda:\fg_0\to\mK$, the action $\pi$ of $\fg$ on $S(\mg_{\minus})$ defined by 
\[
\pi(X) Y:= \mu \left( \sigma^{\minus 1} \otimes \id (X( \sigma(Y) \otimes 1)) \right), \qquad Y\in S(\mg_{\minus}),
\]
is given by
\begin{itemize}
\item $X \in \mg_{\minus} \qquad \pi(X) = X+\sum\limits_{\substack{K>0}}\frac{(-1)^{\abs{K}}B_{\abs{K}}}{K!} s^K_{\mg_{\minus}}(X) \;\partial^K $
\item $X \in \mg_0 \qquad \pi(X) =  \lambda(X)+\sum_{i=1}^{r+s}[X,X_i]\partial^i $
\item $X \in \mg_{\plus} $
 \begin{align*} \pi(X) &= 
 \sum_{i=1}^{r+s} \sum_{j =0}^l  \sum\limits_{\substack{K_1>0}} \cdots \sum\limits_{\substack{K_j>0}} W_{i,K_1,\ldots K_j}(X)  \partial^{K_j}\cdots \partial^{K_1} \partial^i    \\
&+ \sum_{i=1}^{r+s} \sum_{j =0}^{l-1}  \sum\limits_{\substack{K_1>0}} \cdots \sum\limits_{\substack{K_j>0}} \lambda(H_{i,K_1,\ldots K_j}(X)) \partial^{K_j}\cdots \partial^{K_1} \partial^i  
\end{align*} 
\end{itemize}
where the $X_i$ form a homogeneous basis of $\mg_{\minus}$.
\end{theorem}
This theorem implies Theorem \ref{ThmGen}(ii). The expressions for $\pi(X)$ show that $\pi(X)\in \widehat{\wA}(\fg_{\minus})$. As~$\fg$ is finite dimensional, only a finite number of terms in its $\mZ$-grading are non-zero. This implies that $s_{\fg_{\minus}}^K(X)=0$ for any $X\in\fg_{\minus}$, for $|K|$ sufficiently large and likewise only a finite number of $H_{i,K_1,\ldots K_j}(X)$ (and hence $W_{i,K_1,\ldots K_j}(X)$) is non-zero for $X\in \fg_{\plus}$. Consequently $\pi(X)\in {\wA}(\fg_{\minus})$ for all $X\in\fg$.
\begin{proof}[Proof of Theorem \ref{Polynomial realisation of a graded Lie algebra}]
Let $X$ be an element of $\mg_{\minus}$, $Y$ an element of $S(\mg_{\minus})$ and extend the basis $(X_i)$ to a basis of $\fg$. Applying Theorem \ref{Berezin theorem}, we obtain
\[
X \sigma(Y)  = \sigma(XY)+ \sum_{K\in \mN^{m|n}, K>0} \frac{(-1)^{\abs{K}}B_{\abs{K}}}{K!} \sigma  \left( s^K_{\mg}(X)  \; \partial^K Y\right).
\]
Since $Y \in S(\mg_{\minus})$, the derivative $\partial^i Y =0$ for all $i>r+s$. Therefore we can restrict to $K\in\mN^{r|s}$ in the summation, hence
\[
 \pi(X)Y = XY+\sum\limits_{\substack{K>0}}\frac{(-1)^{\abs{K}}B_{\abs{K}}}{K!} s^K_{\mg_{\minus}}(X) \;\partial^K Y.
\]

Now, let $X$ be a homogeneous element of $\mg_0.$ Using Lemma~\ref{lemma Lie bracket of sigma}, we find
\begin{align*}
X\sigma(Y) \otimes 1 = [X,\sigma(Y)] \otimes 1 + (-1)^{\abs{X}\abs{Y}} \sigma(Y) X \otimes 1 \\
= \left(\sum_{i=1}^{r+s} \sigma\left([X,X_i] \partial^i  Y \right) + \lambda(X) \sigma(Y)\right) \otimes 1,
\end{align*}
where we again restricted our summation to elements in $\mg_{\minus}$ since the partial derivatives of $Y$ with respect to elements in $\mk$ are zero. We also used that $\lambda(X)\not=0$ implies that $\abs{X}=0$ since $\lambda$ is an even morphism.
This proves the claim for $X$ in $\mg_0$.

Finally, let $X$ be a homogeneous element of $\mg_{\plus}$. 
From now on we will also drop the $X$ in $W_{i,K_1, \ldots K_j}(X)$ and $H_{i,K_1, \ldots K_j}(X)$ for ease of notation.
Using Lemma~\ref{lemma Lie bracket of sigma} and Corollary \ref{Lemma rewrite of theorem Berezin}, we get
\begin{align*}
X \sigma(Y) \otimes 1 &= \sigma( \sum_{i=1}^{r+s} [X,X_i] \partial^i  Y) \otimes 1 + (-1)^{\abs{X}\abs{Y}} \sigma(Y) X \otimes 1 \\
&=  \sigma( \sum_{i=1}^{r+s} W_i \partial^i  Y) \otimes 1 + \sum_{i=1}^{r+s} (-1)^{(\abs{X}+\abs{X_i})\abs{\partial^i Y}} \sigma(\partial^i  Y) H_i \otimes 1 \\ 
&+ \sum_{i=1}^{r+s} \sum\limits_{\substack{K_1>0}} \frac{(-1)^{\abs{K_1}}C_{\abs{K_1}}}{K_1!} \sigma(s^{K_1}_{\mg_{\minus}}(H_i) \; \partial^{K_1} \partial^i   Y) \otimes 1 \\
&=\sum_{i=1}^{r+s} \sigma (W_i \partial^i  Y) \otimes 1 
+ \sum_{i=1}^{r+s} \lambda(H_i) \sigma(\partial^i  Y) \otimes 1  \\ & + \sum_{i=1}^{r+s} \sum\limits_{\substack{K_1>0}} \sigma\left((W_{i,K_1}+H_{i,K_1}) \; \partial^{K_1} \partial^i   Y\right) \otimes 1.
\end{align*}

We will now repeatedly apply Corollary \ref{Lemma rewrite of theorem Berezin} to the part in $H_{i,K_1, \ldots, K_j} $. This procedure finishes after $l$ steps, since $H_{i,K_1,\ldots K_l}=0.$
Thus
\begin{align*}
X \sigma(Y) \otimes 1 &=\sum_{i=1}^{r+s} \sigma (W_i \partial^i  Y) \otimes 1 + \sum_{i=1}^{r+s} \lambda(H_i) \sigma(\partial^i  Y) \otimes 1 \\
&+ \sum_{i=1}^{r+s} \sum_{j =1}^l  \sum\limits_{\substack{K_1>0}} \cdots \sum\limits_{\substack{K_j>0}} \sigma(W_{i,K_1,\ldots K_j}  \; \partial^{K_j}\cdots \partial^{K_1} \partial^i   Y) \otimes 1 \\
&+ \sum_{i=1}^{r+s} \sum_{j =1}^{l-1}  \sum\limits_{\substack{K_1>0}} \cdots \sum\limits_{\substack{K_j>0}} \lambda(H_{i,K_1,\ldots K_j} ) \sigma(\partial^{K_j}\cdots \partial^{K_1} \partial^i   Y)  \otimes 1.
\end{align*}
 This concludes the proof. 
\end{proof}

\begin{remark}{\rm
If $\mg$ is $3$-graded, $\pi:\fg\to\wA(\fg_{\minus})$ of Theorem \ref{Polynomial realisation of a graded Lie algebra} simplifies to
\begin{enumerate}[(i)] \label{three graded realisation}
\item $X \in \mg_{\minus 1} \qquad \pi(X) = X$
\item $X \in \mg_0 \qquad \pi(X) =  \lambda(X)+\sum_{i=1}^{r+s}[X,X_i]\partial^i  $
\item $X \in \mg_{\plus 1} \qquad \pi(X) = \sum_{i=1}^{r+s} \lambda([X,X_i])\partial^i  +\frac{1}{2}\sum_{i,j=1}^{r+s} [[X,X_i],X_j]\partial^j  \partial^i .$
\end{enumerate}}
\end{remark}

\section{Bessel operators for Jordan superpairs}\label{secBessel}
In this section we consider $\mK=\mR$.

\subsection{The general case}
Consider a real Jordan superpair $V=(V_{\plus},V_{\minus})$. Remark \ref{three graded realisation} then yields a representation of $\fg:={\rm TKK}(V)$, as in \eqref{TKKV}, on $S(V_{\minus})$ for any character $\lambda:\mathfrak{str}(V)\to \mR$. This extends to a representation on $\cO_{\mA(V_{\minus}^\ast)}$, where we identify $S(V_{\minus})$ with the polynomials on~$\mA(V_{\minus}^\ast)$.

Using the operators introduced in Subsection \ref{subsecJSP}, we can rewrite that representation
$$\pi: {\rm TKK}(V)=V_{\plus}\oplus \mathfrak{str} (V)\oplus V_{\minus}\;\;\;\to\;\;\; \wA(V_{\minus})\,\subset\, \Gamma(\cD_{\mA(V_{\minus}^\ast)})$$
into the following form:
\begin{enumerate}[(i)'] 
\item $\pi(0,0,u) = u\qquad\qquad\quad\qquad\qquad\qquad\qquad$ for $u\in V_{\minus}$
\item $\pi(0,D_{x,y},0) =  \lambda(D_{x,y})+\sum_{i}D^\ast_{x,y}(e_i)\partial^i \qquad\qquad \qquad$ for $x\in V_{\plus}$, $y\in V_{\minus}$
\item $\pi(v,0,0)= \sum_{i} \lambda(D_{v,e_i})\partial^i -\sum_{i,j} (-1)^{\abs{v}(\abs{e_i}+\abs{e_j})}P_{e_i,e_j}(v) \partial^j \partial^i\qquad$ for $v\in V_{\plus}$.
\end{enumerate}
Here $(e_i)_i$ is a homogeneous basis of $V_{\minus}$ and $\partial^i\in \wA(V_{\minus})\subset \Gamma(\cD_{\mA(V_{\minus}^\ast)})$ the corresponding partial derivatives. For each $x\in V_{\plus}$ the expression for $\pi(x,0,0)$ in (iii)' gives a differential operator on the affine supermanifold $\mA(V_{\minus}^\ast)$, which is of the same parity as $x$, in case~$x$ is homogeneous.
We will use this expression to define the Bessel operator. This even operator is a global differential operator on $\mA(V_{\minus}^\ast)$ taking values in the super vector space $V_{\plus}^\ast$.
\begin{definition}
Consider a Jordan superpair $V=(V_{\plus},V_{\minus})$ and a character $\lambda \colon \mathfrak{str}(V) \to \mR$. For any $u,v \in V_{\minus}$ we define $\lambda_{u} \in V_{\plus}^\ast$ and $\widetilde{P}_{u,v} \in  V_{\minus}\otimes V_{\plus}^\ast $ by 
$$\lambda_{u}(x) = -\lambda(D_{x,u})\quad\mbox{and}\quad \widetilde{P}_{u,v}(x) := (-1)^{\abs{x}(\abs{u}+\abs{v})}P_{u,v}(x)\qquad\mbox{for all } x\in V_{\plus}.$$
Then we define the \textbf{Bessel operator}
\[\cB_\lambda \in \left(\Gamma (\cD_{\mA(V_{\minus}^\ast)})\otimes V_{\plus}^\ast\right)_{\oa},\quad\;\,\mbox{as}\quad\quad
\cB_\lambda= \sum_i \lambda_{e_i} \partial^i + \sum_{i,j}\widetilde{P}_{e_i,e_j} \partial^j\partial^i.
\]
\end{definition}
In particular, by construction, we find
\[\cB_\lambda(x)=-\pi(x,0,0)\qquad \forall x\in V_{\plus} .\]

These Bessel operators and the representation in (i)'-(iii)' can be viewed as a generalisation to the Jordan superpair setting of the construction in~\cite{BesselOperators}, as we will argue in the next subsection.
Proposition 1.4 in~\cite{BesselOperators} generalises to our setting.
\begin{proposition}
Set $M=\mA(V_{\minus}^\ast)$. The Bessel operator satisfies the following:
\begin{enumerate}[(i)]
\item The family of operators  $\mathcal{B}_\lambda (x) \in \Gamma(\cD_M) $ for $x \in V_{\plus}$, supercommutes for fixed $\lambda$. 
\item For any $\phi,\psi\in \cO_M(U)$, with some open $U\subset M_0$, we have the product rule
\[
\mathcal{B}_\lambda(\phi\psi) = \mathcal{B}_\lambda(\phi) \psi + \phi \mathcal{B}_\lambda(\psi) + 2 \sum_{i,j} (-1)^{\abs{e_i}\abs{\phi}} \widetilde{P}_{e_i,e_j}  \partial^j  (\phi) \partial^i (\psi).
\]
\end{enumerate}
\end{proposition}
\begin{proof}
For the first statement, it suffices to note that for $x_1$ and $x_2$  in $V_{\plus}$  we have \[[\pi(x_1),\pi(x_2)]=0,\]
which follows by construction for a representation of a $3$-graded Lie superalgebra.

For the second statement, we have
\begin{align*}
\mathcal{B}_\lambda ( \phi \psi)&=  \sum_i \lambda_{e_i} \partial^i  (\phi \psi) +  \sum_{i,j}  \widetilde{P}_{ e_i,e_j} \partial^j \partial^i   (\phi\psi)
\\
&=\sum_i  \lambda _{e_i} \partial^i   (\phi) \psi+  \sum_{i,j}   \widetilde{P}_{ e_i,e_j} \partial^j (\partial^i   (\phi))\psi
\\ & +  \sum_i \lambda_{e_i} (-1)^{\abs{e_i}\abs{\phi}} \phi \partial^i   ( \psi) +  \sum_{i,j} (-1)^{(\abs{e_i}+\abs{e_j})\abs{\phi}}  \widetilde{P}_{ e_i,e_j} \phi \partial^j \partial^i   (\psi)
\\ &
 + \sum_{i,j} (-1)^{\abs{e_i}\abs{\phi}}  \widetilde{P}_{ e_i,e_j} \partial^j  (\phi)\partial^i  (\psi)
 \\ &
  + \sum_{i,j} (-1)^{\abs{e_j}(\abs{e_i}+\abs{\phi})}  \widetilde{P}_{ e_i,e_j} \partial^i  (\phi)\partial^j  (\psi)
\\
&= \mathcal{B}_\lambda(\phi) \psi + \phi \mathcal{B}_\lambda(\psi) + 2 \sum_{i,j} (-1)^{\abs{e_i}\abs{\phi}} \widetilde{P}_{ e_i,e_j}\partial^j  (\phi) \partial^i (\psi),
\end{align*}
where we used $P_{e_i,e_j}=(-1)^{\abs{e_i}\abs{e_j}}P_{e_j,e_i}$.
\end{proof}

\begin{remark}
If we consider realisation (i)'-(iii)' for the simple Jordan superalgebras D in Section 3 of \cite{Kac}, we obtain the Fourier transform of the realisations of the orthosymplectic superalgebras constructed in Theorem~5.7 of \cite{CSS}.
\end{remark}

\subsection{A special case}

Now assume that $V=(J,J)$, as in Subsection \ref{secJalg}, with $J$ a unital simple Jordan {\em algebra}. We will use concepts and nomenclature as in \cite{FK}. We then make the extra assumption on $J$ that the eigenspace for eigenvalue $+1$ of the {\em Cartan involution} $\alpha$ of $J$ is a simple Jordan algebra. We also consider the symmetric bilinear form $\tau:J\times J\to \mR$ known as the trace form, which is non-degenerate under the above assumptions. Furthermore we define the symmetric bilinear form $(\cdot|\cdot)=\tau(\cdot,\alpha\cdot)$, which is positive definite. We denote the dimension of $J$ by $n$ and its rank by $r$.

In order to compare the realisation (i)'-(iii)' with the one in \cite{BesselOperators}, we need to adjust to the convention in \cite{BesselOperators}, which considers a realisation of $\mathfrak{co}(J)$ on $\cO_{\mA(J)}$, rather than on $\cO_{\mA(J^\ast)}$. Therefore we define an isomorphism of vector spaces 
$$D: J\to J^\ast;\qquad v\mapsto \tau(v,\cdot) \qquad\forall v\in J, $$
which extends to an isomorphism $D: S(J)\;\tilde\to\; S(J^\ast)$. The representation $\pi$ on $S(J)$ in (i)'-(iii)' then leads to one on $S(J^\ast)$, defined as $D\circ \pi\circ D^{\minus 1}$, which we also denote by $\pi$. This yields
\begin{enumerate}[(i)'']
\item $\pi(0,0,u)=\tau(u,\cdot)$
\item $\pi(0,D_{x,y},0)=\lambda(D_{x,y})-\sum_{i=1}^{n}\tau(D_{y,x}(e_i),\cdot)\partial^i$
\item $\pi(v,0,0)=\sum_{i=1}^{n} \lambda(D_{v,e_i})\partial^i-\sum_{i,j}\tau( P_{e_i,e_j}(v),\cdot)\partial^j\partial^i$
\end{enumerate}
for $u,x,y,v\in J$, where the partial derivatives $\partial^i\in\Gamma(\cD_{\mA(J)})$ are taken with respect to the basis $\tau(e_i,\cdot)$ of $J^\ast$.

Now we make the further assumption that the character $\lambda:\mathfrak{str}(J)\to \mR$ is of the form 
$$\lambda\;=\;-\frac{r}{2n}\lambda_0\Tr,$$ with $\Tr$ the trace of operators on $J$ and $\lambda_0\in\mR$. The motivation to include the factor $r/2n$ comes from the observation that expression \eqref{expression D for Jordan algebra} of the operator $D_{x,y}$ and Proposition III.4.2 in~\cite{FK} imply that for any $D_{x,y}\in\mathfrak{str}(J)$
\[
\Tr(D_{x,y})=2\Tr(L_{x\ast y}) =\frac{2n}{r} \tau(x,y).
\]
Furthermore we also have $\tau(D_{x,y} u, v) =  \tau(u, D_{y,x} v) $
and $\tau(P_{x,y}u,v) =  \tau(P_{x,y}v,u),$ which follows easily if one uses the associativity of the trace form $\tau$, see Proposition II.4.3 in \cite{FK}, and the expressions given in \eqref{expression D for Jordan algebra} for $D_{x,y}$ and $P_{x,y}$.

The Cartan involution $\theta$, which is an involutive automorphism of $\mathfrak{co}(J)$, see e.g. Section 2.1.1 in \cite{BesselOperators} is given by
\[
\theta(u,D_{x,y},v)=(-\alpha v, -D_{\alpha y, \alpha x},-\alpha u).
\]
Then we rewrite the above representation in terms of $(\cdot|\cdot)$ and compose it with the Cartan involution, yielding another representation $\pi'$ which takes the form we describe below.

\begin{scholium}\label{schol}
The representation in Remark \ref{three graded realisation} contains as a special case the following situation.
For any real Jordan algebra $J$, with assumptions as above, the Lie algebra $\mathfrak{co}(J)$  admits a representation~$\pi'$ on smooth functions on the manifold $J$, given by
\begin{enumerate}[(I)]
\item $\pi'(u,0,0)=-(u|\cdot)$
\item $\pi'(0,D_{x,y},0)=\frac{r\lambda_0}{2n} \Tr(D_{\alpha y,\alpha x})+\sum_{i=1}^{n}\tau(e_i, D_{\alpha y,\alpha x}\cdot)\partial^i$
\item $\pi'(0,0,v)=\lambda_0\sum_{i=1}^{n}  ( v|e_i)\partial^i+\sum_{i,j}(v|P_{e_i,e_j}\cdot)\partial^j\partial^i$,
\end{enumerate} for all $u,x,y,v\in J$.
This is the representation occurring in Section~2.2 in~\cite{BesselOperators}. Note that the action of $\mathfrak{co}(J)_{\pm 1}$ needs to be multiplied with an auxiliary constant $\mp \imath$ in order to get the exact same expressions. Furthermore we should point out that in \cite{BesselOperators} it is shown that this is not just a representation on functions on $J$, but that it also restricts to functions on certain orbits of the structure group.
\end{scholium}

\section{A realisation of the Lie superalgebra $D(2,1;\alpha)$}\label{secD}
In this section we give polynomial realisations for the one parameter family of Lie superalgebras $D(2,1;\alpha)$ and discuss the reducibility of the corresponding representation on polynomials. We also comment on the other exceptional basic classical Lie superalgebras. In this section we set $\mK=\mC$.

\subsection{Realisations of the exceptional Lie superalgebras}

The Weyl superalgebra $\wA(V)$, as a super vector space, inherits a $\mZ\times \mZ$-grading from the natural gradings on $S(V)$ and $S(V^\ast)$ by equation~\eqref{ASS}.

\begin{proposition}$\mbox{}$
\begin{enumerate}[(i)]
\item For any $\alpha\in\mC\backslash\{0,-1\}$, the Lie superalgebra $D(2,1;\alpha)$ admits a one parameter family of realisations in $\wA:=\wA(\mC^{2|2})$, which are contained in \[\wA_{0,0}\oplus\wA_{1,0}\oplus\wA_{1,1}\oplus\wA_{0,1}\oplus\wA_{1,2}.\]
One of those realisations is inside $\wA_{1,0}\oplus\wA_{1,1}\oplus\wA_{1,2}.$
\item The Lie superalgebra $G(3)$ admits a one parameter family of realisations in $\wA:=\wA(\mC^{1|7})$, which are contained in
\[\wA_{0,0}\oplus\wA_{1,0}\oplus\wA_{1,1}\oplus\wA_{0,1}\oplus\wA_{1,2}\oplus \wA_{0,2}\oplus \wA_{1,3}.\]
One of those is inside $\wA_{1,0}\oplus\wA_{1,1}\oplus\wA_{1,2} \oplus \wA_{1,3}.$

\item The Lie superalgebra $F(4)$ admits a one parameter family of realisations in $\wA:=\wA(\mC^{6|4})$, which are contained in $$\wA_{0,0}\oplus\wA_{1,0}\oplus\wA_{1,1}\oplus\wA_{0,1}\oplus\wA_{1,2}.$$
One of those realisations is inside $\wA_{1,0}\oplus\wA_{1,1}\oplus\wA_{1,2}.$
\end{enumerate}
\end{proposition}
\begin{proof}
Consider a simple $\mZ$-graded Lie superalgebra $\fg=\oplus_{|j|\le d}\,\fg_{j}$, where~$\fg_0$ is the direct sum of its (even) centre $\mathfrak{z}$ and some simple Lie superalgebras and set $k=\dim\mathfrak{z}$. Theorem \ref{Polynomial realisation of a graded Lie algebra} then implies that $\fg$ has a $k$ parameter family of realisations in $\wA:=\wA(\fg_{\minus})$, which is contained in
\[(\oplus_{0\le i \le d-1}\wA_{1,i})\;+\;(\wA_{0,0}\oplus\wA_{1,1})\;+\;(\oplus_{1\le i\le d+1 }\wA_{1,i}\;\oplus\;\oplus_{1\le j\le d}\wA_{0,j}).\]
In the particular case that $\lambda=0$, the realisation is actually contained in 
$$(\oplus_{0\le i \le d-1}\wA_{1,i})\;+\;\wA_{1,1}\;+\;(\oplus_{1\le i\le d+1 }\wA_{1,i}).$$

Now $D(2,1;\alpha)$ has a 3-term grading with $\fg_0=\mathfrak{osp}(2|2)\oplus\mC$ and $\mathfrak{z}=\mC$, which will be considered explicitly in Subsection \ref{subsecD}, the results follow then from using $d=1$ in the above formulae.

The Lie superalgebra $G(3)$ has a 5-term grading with $\fg_0\cong G(2)\oplus\mC$ and $\mathfrak{z}=\mC$, see Section 2.19 in \cite{dictionary}, where~$\mg_{\minus 2}$ is one dimensional and even and $\mg_{\minus 1}$ is seven dimensional and purely odd. Setting $d=2$ in the above formulae, we get the result for $G(3)$.

For the last case, we remark that from \cite{Kac} it follows that $F(4)$ is the Tits-Kantor-Koecher Lie superalgebra associated with the Jordan superalgebra $F$ which has dimension $(6|4)$. The 3-term grading coming from this ${\rm TKK}$-construction, can be derived from Proposition 1(I) in~\cite{Kac}. The even and odd roots of $F(4)$ are given by 
\[
\Delta_0 = \{ \pm \delta, \pm \epsilon_{i}\pm \epsilon_{j}, \pm \epsilon_i \} \quad i,j \in \{1,2,3\}, \qquad \qquad \Delta_1 = \{  \frac{1}{2}(\pm \delta \pm \epa \pm \epb \pm \epc) \}.
\]

Then the 3-term grading is given by
\begin{align*}
\fg_{\plus} &= \{ X_{\epa},X_{\delta}, X_{\epa\pm\epb},X_{\epa\pm \epc},X_{\frac{1}{2}(\delta+\epa\pm\epb\pm\epc)}\} \\
\fg_{\minus} &=\{ X_{-\epa},X_{-\delta}, X_{-\epa\pm\epb},X_{-\epa\pm \epc},X_{-\frac{1}{2}(\delta+\epa\pm\epb\pm\epc)}\} \\
\fg_0 &=  \{   H_\delta, H_{\epa}, H_{\epb}, H_{\epc},X_{\pm\epb},X_{\pm\epc},X_{\pm\epb\pm\epc} X_{\frac{1}{2}(\delta-\epa\pm\epb\pm\epc)},X_{\frac{1}{2}(-\delta+\epa\pm\epb\pm\epc)} \}.
\end{align*}
We have that $\mg_0= \mathfrak{osp}(2|4)\oplus \mathfrak{z}$. The center is given by $\mathfrak{z}=H_\delta+H_\epa$, hence one-dimensional. Again setting $d=1$, we also obtain the last result.
\end{proof}

The specific form of the realisation of $D(2,1;\alpha)$ inside $\wA_{1,0}\oplus\wA_{1,1}\oplus\wA_{1,2}$ is given in the following proposition.
\begin{proposition}
Consider the differential operators 
\begin{align*}\mE&=x\partial_x+y\partial_y+\eta\partial_\eta+\theta\partial_\theta, & \Delta&=\partial_x\partial_y+\partial_\eta\partial_\theta,\\
\mE_\alpha&=\alpha x\partial_x+y\partial_y+\alpha\eta\partial_\eta+\theta\partial_\theta, & \Delta_\alpha&=\partial_x\partial_y+\alpha\partial_\eta\partial_\theta
\end{align*}
on $\mA^{2|2}$. A realisation of $D(2,1;\alpha)$ is given by the operators
\begin{gather*} x,y,\theta,\eta, \\
 \theta\partial_\theta-\eta\partial_\eta, x\partial_x-y\partial_y, \mE,\theta\partial_\eta,\eta\partial_\theta, \eta\partial_y+\alpha x\partial_\theta, \theta\partial_x-y\partial_\eta,\theta\partial_y-\alpha x\partial_\eta,\eta\partial_x+y\partial_\theta
\end{gather*}
and
\[ \mE\partial_x-y\Delta,\;\, \mE\partial_y-x\Delta_\alpha,\;\,\mE_\alpha\partial_\theta+\eta\Delta,\;\,\mE_\alpha\partial_\eta-\theta\Delta_\alpha.\]
\end{proposition}
This will be obtained as a special case of the realisations considered in the next subsection.

\subsection{The realisations of $D(2,1;\alpha)$}\label{subsecD}

These~$17$-dimensional Lie superalgebras of rank $3$ are deformations of $D(2,1)=\mathfrak{osp}(4|2)$. We will always asume that $\alpha \in \mC$ is different from $0$ and $-1$. In that case $D(2,1;\alpha)$ is simple. The even and odd roots are given by
\[
\Delta_0 = \{ \pm 2\dea, \pm 2\deb , \pm 2\dec \} \qquad \qquad \Delta_1 = \{ \pm \dea \pm \deb \pm \dec \}.
\]
With respect to the simple root system 
\[
\Pi = \{ 2\deb, \dea-\deb-\dec, 2\dec \}
\]
the Cartan matrix is given by
\begin{align*}
\begin{pmatrix}
2 & -1 & 0 \\
-1 & 0 & -\alpha \\
0 &-1 & 2
\end{pmatrix},
\end{align*}
see Sections 4.2  and 5.3.1  in~\cite{Musson}.
 For the values $\alpha= 1,-2,$ or $\frac{\minus 1}{2}$, it is isomorphic to $\mathfrak{osp}(4|2)$.
For more information on irreducible representations of $D(2,1;\alpha)$, see \cite{Joris}.

The bilinear form considered on $\fh^\ast$ is given by
\[
\langle \delta_i ,\delta_j \rangle=\delta_{ij}, \text{ for } i, j \in \{1,2,3\}.
\]

A basis of the Cartan subalgebra is given by $ \{H_\dea, H_\deb, H_\dec \},$ such that $\beta(H_{\delta_i})=\langle \beta,\delta_i\rangle$ for all $\beta\in\fh^\ast$. We complete this basis to a basis of $D(2,1;\alpha)$ by adding non-zero root vectors $X_\gamma$ for any $\gamma\in \Delta$, which we normalise such that
\begin{itemize}
\item $ [X_{2\deb}, X_{-2 \deb}]=H_\deb $
\item $ [X_{\dea-\deb-\dec}, X_{-\dea +\deb+\dec}]=-\frac{1+\alpha}{2} H_\dea -\frac{1}{2} H_\deb-\frac{\alpha}{2} H_\dec $
\item $ [X_{2\dec}, X_{-2\dec}]=H_\dec.$
\end{itemize}

Consider the following 3-grading of $D(2,1;\alpha)$:
\begin{align*}
\fg_{\plus} &= \{ X_{2\dec} ,  X_{2\deb}, X_{-\dea+\deb+\dec} ,X_{\dea+\deb+\dec}\} \\
\fg_{\minus} &=\{ X_{-2\dec} ,  X_{-2\deb},X_{\dea-\deb-\dec}, X_{-\dea-\deb-\dec} \}  \\
\fg_0 &=  \{   H_\dea, H_\deb,H_\dec, X_{2\dea}, X_{-\dea+\deb-\dec},X_{\dea+\deb-\dec},X_{-2\dea}, X_{\dea-\deb+\dec},X_{-\dea-\deb+\dec} \}.
\end{align*}
Hence we have $\fg_0=\mathfrak{osp}(2|2)\oplus\mC$, where the ideal $\mC$ is the centre of $\fg_0$ and $\mathfrak{osp}(2|2)$ is simple. We choose~$h:=H_{\delta_2}+H_{\delta_3}\in \fh^\ast\subset\fg_0$ and the centre of $\fg_0$ is given by $\mC h$. Hence, there is a bijection between characters $\lambda:\fg_0\to \mC$ and $\mC$, which we normalise by $\lambda\mapsto \lambda(h)$.

We consider the realisation of $D(2,1;\alpha)$ in $\wA(\Span_{\mC}(x,y,\theta,\eta)),$ where~$x,y$ are even and $\theta, \eta$ are odd, given by Remark~\ref{three graded realisation}. We add the character (complex number) $\lambda$ in the notation. For $\fg_{\minus}$ we have
\begin{align*}
\pi_\lambda(X_{-2\dec})&=x, & \pi_\lambda(X_{-2\deb})&=y ,\\ 
\pi_\lambda(X_{\dea-\deb-\dec})&=\theta, &\pi_\lambda(X_{-\dea-\deb-\dec})&=\eta.
\end{align*}
For $\fg_{0}\cong\mC\oplus\mathfrak{osp}(2|2)$ we have
\begin{align*}
   \pi_\lambda(H_\dea)&=\theta\partial_{\theta}-\eta \partial_{\eta}, & \pi_\lambda(H_{\deb})&=\lambda- 2y \partial_{y}-\theta \partial_{\theta} - \eta \partial_{\eta} \\
    \pi_\lambda(H_{\dec})&= \frac{\lambda}{\alpha} -2x\partial_{x}-\theta \partial_{\theta} - \eta \partial_{\eta}, \\
    \pi_\lambda(X_{2\dea})&=(1+\alpha) \theta \partial_{\eta},&\pi_\lambda(X_{-2\dea})&=(1+\alpha) \eta \partial_{\theta},\\
    \pi_\lambda(X_{-\dea+\deb-\dec})&=- \eta \partial_{y}-\alpha x \partial_{\theta},&\pi_\lambda(X_{\dea-\deb+\dec})&= \theta \partial_{x}-y \partial_{\eta},\\
    \pi_\lambda(X_{\dea+\deb-\dec})&= \theta \partial_{y}-\alpha x \partial_{\eta},&\pi_\lambda(X_{-\dea-\deb+\dec})&= -\eta \partial_{x}-y \partial_{\theta}.\\
\end{align*}
For $\fg_{\plus}$ we finally find
\begin{align*}
\pi_\lambda(X_{2\dec})&= \left(\frac{\lambda}{\alpha}-x\partial_{x}-\theta\partial_{\theta}-\eta\partial_{\eta} \right)\partial_x + y \partial_{\eta}\partial_{\theta}  , \\
\pi_\lambda(X_{2\deb})&=\left(\lambda - y \partial_{y} - \theta\partial_{\theta} - \eta \partial_{\eta} \right)\partial_{y} +\alpha x \partial_{\eta}\partial_{\theta},\\
 \pi_\lambda( X_{-\dea+ \deb+\dec})&=(-\lambda + \alpha x \partial_{x} +  y \partial_{y} + (1+\alpha)\eta \partial_{\eta} )\partial_{\theta} + \eta \partial_{x}\partial_{y},\\
 \pi_\lambda(X_{\dea+\deb+\dec})&= \left( \lambda -\alpha x\partial_{x} - y \partial_{y} - (1+\alpha) \theta\partial_{\theta}\right)\partial_{\eta} + \theta \partial_{x}\partial_{y}.
\end{align*}

The restriction of the canonical representation of $\wA_{2|2}$ on $S(\mC^{2|2})$ to $U(\fg)$, seen as a subalgebra through $\pi_\lambda$, leads to a representation of $\fg=D(2,1;\alpha)$ on $S(\mC^{2|2})$, which we also denote by~$\pi_\lambda$. By construction and Scholium \ref{schol}, this is an analogue for superalgebras of the conformal representations considered in \cite{BesselOperators}. Another key step in the construction in {\it op. cit.} is the fact that for certain values of the parameter $\lambda$, the operators in the realisation are tangential to specific orbits of the structure group on the Jordan algebra. Consequently, the representation on functions on $J$ is not irreducible and the representation of interest is a factor module of $\cC^\infty(J)$. The set of parameters for which this occurs is directly linked to the Wallach set, see e.g. Theorem 1.12 of \cite{BesselOperators}. This motivates the question for which $\lambda$, the representation $\pi_\lambda$ is irreducible in our example for $D(2,1;\alpha)$.

\begin{proposition}\label{propsimple}
The representation $\pi_\lambda$ of $D(2,1;\alpha)$ on $S(\mC^{2|2})$ is irreducible if and only if 
\[\lambda \not\in \mN \quad\mbox{and}\quad \lambda/\alpha\not\in\mN.\]
If either $\lambda \in \mN $ or $\lambda/\alpha\in\mN,$ the representation is indecomposable but not irreducible.
\end{proposition}
\begin{proof}
Set $D(2,1;\alpha)=\fg$. First we note that all modules are weight modules and that the weight corresponding to the constants in $\cP:=S(\mC^{2|2})$ appears with multiplicity one. If $\cP$ would be the direct sum of two $\fg$-modules, the space of constants would hence belong to precisely one of them. However, it is clear that $U(\fg_{\minus})$-action on $1$ generates $\cP$, leading to a contradiction. Therefore the module is indecomposable. It is simple if and only if for any $P\in \cP$ we have $1\in U(\fg)\cdot P$.

Let $P$ be a homogeneous polynomial of degree $l>0$ in $\mC [x,y,\theta,\eta]$.
By a lengthy but straightforward calculation, one can show that $\mg_{\plus}$ acting trivially on $P$ forces $P$ to be zero unless $\lambda \not= i$ or $\lambda \not= i \alpha $ for some $i< l$, $i \in \mN$.

First assume that both $\lambda$ and $\lambda/\alpha$ are not in $\mN$. For any homogeneous polynomial of degree $l$, there exists an element $v \in \mg_{\plus}$ such that $\pi_\lambda(v)P$ is a non-zero homogenous polynomial of degree $l-1$. By induction, we can find $v_1, \ldots, v_l \in \mg_{\plus}$ such that $\pi_\lambda(v_1)\pi_\lambda(v_2)\cdots \pi_\lambda(v_l)P$ is a non-zero constant. For $P$ an arbitrary polynomial we can consider the homogeneous polynomial $P_{{\rm max}}$ such that the polynomial $P-P_{{\rm max}}$ is of strictly lower degree than $P$. The above argument yields an element of $U(\fg_{\plus})$ which annihilates $P-P_{{\rm max}}$ and maps $P_{{\rm max}}$ (and hence P) to a non-zero constant. We thus find that the representation is irreducible.

On the other hand, one can check directly that $\mg_{\plus}$ acts trivially on 
\[
P:= ax+b y+ c\theta+d\eta,
\]
if $\lambda =0$, on
\[
P:= ay^{l+1}+b y^{l}\theta + c y^{l}\eta + d(y^{l-1} \theta\eta-\frac{\alpha}{l} xy^{l}),
\]
if $\lambda = l>0$ and on
\[
P:=a x^{l+1} + b x^{l}\theta +c x^{l}\eta + d(x^{l-1} \theta\eta-\frac{1}{l} x^{l}y),
\]
if $\lambda = l\alpha$ with $l>0$. Here $a,b,c,d$ are arbitrary complex constants.
From the PBW-Theorem it follows that $U(\fg)\cdot P\cong U(\mg_{\minus})U(\mg_0)U(\mg_{\plus}) \cdot P$. Since $\mg_{\plus}$ acts trivially, it follows that all polynomials in $U(\fg)\cdot P$ have degree higher or equal to $P$. Therefore  $U(\fg)\cdot P$ is a proper submodule of $\cP$.
\end{proof}

We conclude this section by focusing on the specific cases $\lambda=1$ and $\lambda=\alpha$, as in the spirit of the above discussion the top of that module seems the first candidate for the `minimal representation' of $D(2,1;\alpha)$. 

First assume that $\alpha=1$, then the action of $\mathfrak{osp}(2|2)$, the semisimple part of $\fg_0$, on~$\cP$ reduces to the one studied in~\cite{osp}. In particular it was derived that in this case the space~$\cP_2$ of homogeneous polynomials is indecomposable. This self-dual module has a simple socle given by the trivial representation generated by the polynomial $R^2=xy+\eta\theta$. The calculations in the proof of Proposition \ref{propsimple} illustrate that the polynomials of degree 2 which generate the $D(2,1;1)$-submodule of $\cP$ constitute a subspace of codimension 1. This is precisely the radical of the $\mathfrak{osp}(2|2)$-module $\cP_2$, namely the solutions of the Laplace equation. 

Now return to the case~$\lambda=1$ or $\lambda=\alpha$ with $\alpha\not=1$. In this case the structure of the $\fg_0$-module clearly changes. There is no longer a 1-dimensional submodule. But in both cases there is a 5-dimensional submodule which generates the $D(2,1;\alpha)$-submodule of $\cP$. This 5-dimensional submodule is generated either by $R^2=xy+\eta\theta$, if $\lambda=\alpha$, or $R^2_\alpha=\alpha xy+\eta\theta$, if $\lambda=1$. 

\appendix
\section{Three technical lemmata}\label{technical}

In this section we obtain several technical results concerning the operator $s^K_{\fg}\in \End_{\mK}(\fg)$ of Section \ref{Constr}, for $\fg$ a $\mK$-Lie superalgebra of dimension $m|n$ and $K\in\mN^{m|n}$, defined as
\begin{equation}\label{defeqS}s^K_{\mg}(X):=\sigma\left(\widetilde{\ad}_{X_1}^{\bullet k_1} \bullet \cdots\bullet \widetilde{\ad}_{X_{m+n}}^{\bullet k_{m+n}}\right)X\qquad \text{ for all } X \in \mg,\end{equation}
where~$\{X_i,1\le i \le m\}$ constitutes a basis of $\fg_{\oa}$ and $\{X_{m+j,1\le j\le n}\}$ a basis of $\fg_{\ob}$.
 
\begin{lemma}\label{lemma 1}
For $X \in \mg$ and $Y$ in $S(\mg)$, the following holds
\begin{align} \label{equation lemma 1}
\sigma(XY) = \sum_{K\in \mN^{m|n}} \frac{(-1)^{\abs{K}} }{(\abs{K}+1)K!} s^K_{\mg}(X) \sigma( \partial^K Y).
\end{align} 
\end{lemma}
\begin{proof}
By linearity it suffices to consider $Y\in S(\mg)$ of the form 
\begin{align} \label{definition alpha}
 Y=X_1^{\alpha_1}X_2^{\alpha_2} \cdots X_{m+n}^{\alpha_{m+n} },
 \end{align}
 for some $\alpha \in \mN^{m|n}$.
  Put $p=\sum_{i=1}^{m} \alpha_i$ and $q= \sum_{i=m+1}^{m+n}\alpha_i$. Then we write $$Y=Z_1Z_2 \cdots Z_{p+q},$$ where the $Z_i\in \fg$ are defined by
 \begin{align*}
Z_1 =Z_2 = \cdots = Z_{\alpha_1}=&X_{1}, \\
 Z_{\alpha_1+1} =Z_{\alpha_1+2} = \cdots= Z_{\alpha_1+\alpha_2}= & X_2, \\
  \vdots & \\ 
Z_{p+q-\alpha_{m+n}+1} = \cdots = Z_{p+q}= & X_{m+n}.
\end{align*}
Remark that $Z_i$ is even for $i \leq p$ and odd for $i> p$.

Since \eqref{equation lemma 1} is also linear in $X,$ we can assume $X$ to be homogeneous. So, let $X$ be a homogeneous element of $\mg$ and define $p+q$ indeterminates $t_i$, where~$t_i$ is even if $i\leq p$ and odd if $i >p.$ Furthermore we define the indeterminate $t$ to be even if $X$ is even and odd if $X$ is odd. 
Consider the supercommutative algebra $T$ generated by $\{t_i, i=1,\cdots,p+q\}$ and $t$. Then we define the following element of $U(\mg) \otimes T$
\begin{align*}
\kappa(t)= X t+ \sum_{i=1}^{p+q} Z_i t_i.
\end{align*}
By construction we have $\abs{\kappa(t)}=0$, therefore $[W,\kappa(t)]= W\kappa(t)-\kappa(t)W$ for all $W \in U(\mg) \otimes T$.

We will calculate $\frac{\partial }{\partial t} \kappa(t)^{p+q+1} |_{t=0} $ in two different ways and then compare the term in $t_1 \cdots t_{p+q}$.

On the one hand, we have for $j \in \mN$
\begin{align*} 
\kappa(t)^j = \sum_{(\sum_{i=0}^{p+q} r_i)=j} \frac{j!}{r_0!r_1!r_2! \cdots r_p!} \;\sigma \left(X^{r_0}Z_1^{r_1}\cdots Z_{p+q}^{r_{p+q}}\right) \; t_{p+q}^{r_{p+q}} \cdots t_1^{r_1} t^{r_0}.
\end{align*}

Setting $j=p+q+1$, we obtain
\begin{align*}
\frac{\partial }{\partial t} \kappa(t)^{p+q+1} |_{t=0}= 
 \sum_{(\sum_{i=1}^{p+q} r_i)=p+q} (-1)^{\abs{X}} \frac{(p+q+1)!}{r_1!r_2! \cdots r_p!} \;\sigma \left(XZ_1^{r_1}\cdots Z_{p+q}^{r_{p+q}}\right) \; t_{p+q}^{r_{p+q}} \cdots t_1^{r_1}.
\end{align*}
Hence the term in $t_1\ldots t_{p+q}$ is given by
\begin{align}  \label{expression with sigma (X)}
(-1)^{\abs{X}}(p+q+1)! \sigma\left(XZ_1\cdots Z_{p+q}\right) \; t_{p+q} \cdots t_1.
\end{align}

On the other hand, we can also calculate $\frac{\partial }{\partial t} \kappa(t)^{p+q+1} |_{t=0}$ using the expression
\begin{align*}
\kappa (t)^i W = \sum_{s=0}^i \binom{i}{s} \ad_{\kappa(t)}^s (W) \; \kappa(t)^{i-s},
\end{align*}
which holds for all $W$ in $U( \mg)  \otimes T$. Using the super Leibniz rule, we find
\begin{align*} 
&\frac{\partial }{\partial t} \kappa(t)^{p+q+1}
 = \sum_{i=0}^{p+q} \kappa(t)^i (-1)^{\abs{X}}X \kappa(t)^{p+q-i} \nonumber
= (-1)^{\abs{X}}\sum_{s=0}^{p+q} \sum_{i=s}^{p+q} \binom{i}{s} \ad_{\kappa(t)}^s (X ) \kappa(t)^{p+q-s} \nonumber \\
&= (-1)^{\abs{X}}\sum_{s=0}^{p+q} \binom{p+q+1}{s+1} \ad_{\kappa(t)}^s (X ) \quad \sum_{\mathclap{(\sum_{i=0}^{p+q} r_i)=p+q-s}}  \frac{(p+q-s)!}{r_0!r_1!r_2! \cdots r_p!} \;\sigma \left(X^{r_0}Z_1^{r_1}\cdots Z_{p+q}^{r_{p+q}}\right) \; t_{p+q}^{r_{p+q}} \cdots t_1^{r_1} t^{r_0}.
\end{align*}
Setting $t=0$, we obtain
\begin{align}\label{other hand}
(-1)^{\abs{X}} \sum_{s=0}^{p+q}   \sum_{(\sum_{i=1}^{p+q} r_i)=p+q-s}  \frac{(p+q+1)!}{(s+1)!r_1!r_2! \cdots r_p!} \;\ad_{\kappa(0)}^s (X ) \; \sigma \left(Z_1^{r_1}\cdots Z_{p+q}^{r_{p+q}}\right) \; t_{p+q}^{r_{p+q}} \cdots t_1^{r_1}.
\end{align}
We will bring all the terms $t_i$ which are still contained in $\ad_{\kappa(0)}^s$ to the right, so that we can compare it to \eqref{expression with sigma (X)}. This will create many minus signs which we will calculate in several steps.
Let $\{ f(1),\ldots, f(s)\} \subset \{1, \ldots, p+q \}$ be a subset which is ordered, i.e. $f(i)<f(j)$ if $i<j$ and let $\tau$ be a permutation of $S_s$. As we will let $\tau$ act on products which are ordered as in Subsection \ref{PrelSec1}, we can use the notation $|\tau|$ of equation \eqref{sign permutation}. Furthermore we will want to manipulate expressions in a way that ignores the relations between the different $Z_i$. Therefore we consider the supersymmetric algebra $\cZ$ generated by even variables $z_i$ for $1\le i\le p$ and odd variables $z_{p+j}$ for $1\le j\le q$. This comes with an algebra morphism $\xi_\alpha:\cZ\to S(\fg)$ defined by $\xi_\alpha(z_i)=Z_i$. Furthermore we introduce $\sigma_\alpha=\sigma\circ\xi_\alpha:\cZ\to U(\fg)$.

\begin{itemize}
\item
Since $Z_jt_j$ is even, we have
\begin{align*}
[Z_{\tau(f(1))}t_{\tau(f(1))}, [\cdots ,[Z_{\tau(f(s))}t_{\tau(f(s))},X] \cdots ]] =(-1)^s[[\cdots [X,Z_{\tau(f(s))}t_{\tau(f(s))}], \cdots],Z_{\tau(f(1))}t_{\tau(f(1))}]\\
=(-1)^s[[\cdots [X,Z_{\tau(f(s))}], \cdots],Z_{\tau(f(1))}] \;\;t_{\tau(f(1))} \cdots  t_{\tau(f(s))}.
\end{align*}
\item
Denote by $\hat{Z}^{f}$ the product $Z_1Z_2\cdots Z_{p+q}$ after all terms in $\{Z_{f(i)}\,|\,i=1,\cdots,s\}$ are omitted and similarly by $\hat{t}^f$ the product $t_{p+q}\cdots t_1$ after removing $\{t_{f(i)}\,|\,i=1,\cdots,s\}$\begin{align*}
&t_{\tau(f(1))} \cdots  t_{\tau(f(s))}\sigma(\hat{Z}^{f}) \; \hat{t}^f \\
&=\sigma\left(\xi_\alpha(\partial_{z_{\tau(f(1))}}\cdots \partial_{z_{\tau(f(s))}} z_{\tau(f(s))} \cdots z_{\tau(f(1))})\hat{Z}^{f}\right)\hat{t}^ft_{\tau(f(1))} \cdots  t_{\tau(f(s))} \\
&=\sigma_\alpha\left(\partial_{z_{\tau(f(1))}}\cdots \partial_{z_{\tau(f(s))}}z_1 \ldots  \ldots z_{p+q}\right) t_{p+q} \cdots t_1.
\end{align*}
\item Finally
\begin{align*}
\partial_{z_{\tau(f(1))}}\cdots \partial_{z_{\tau(f(s))}} = (-1)^{\abs{\tau}} \partial_{z_{f(1)}} \cdots \partial_{z_{f(s)}}.
\end{align*}
\end{itemize}
Combining these three calculations we conclude
\begin{align*}
[Z_{\tau(f(1))}t_{\tau(f(1))}, [\cdots ,[Z_{\tau(f(s))}t_{\tau(f(s))},X] \cdots ]]\;\sigma(\hat{Z}^{f}) \; \hat{t}^f \\
= (-1)^{\delta(\tau,f)} [[\cdots [X,Z_{\tau(f(s))}], \cdots],Z_{\tau(f(1))}]\; \sigma_\alpha(\partial_{z_{f(1)}} \cdots \partial_{z_{f(s)}} z_1 \cdots z_{p+q}) \; t_{p+q}\cdots t_1,
\end{align*}
where~$\delta(\tau,f) = s + \abs{\tau}$.

Therefore the term of $\frac{\partial}{\partial t} \kappa(t)^{p+q+1} |_{t=0}$ in $t_1 \cdots t_{p+q}$  is given by
\begin{align} \label{A}
(-1)^{\abs{X}}\sum_{s=0}^{p+q} \sum\limits_{\substack{ f\\ \abs{f}=s}} \sum_{\tau} (-1)^{\delta(\tau,f)} \frac{(p+q+1)!}{(s+1)!}  [[\cdots [X,Z_{\tau(f(s))}], \cdots],Z_{\tau(f(1))}] \\ \sigma_\alpha(\partial_{z_{f(1)}} \cdots \partial_{z_{f(s)}} z_1 \cdots z_{p+q}) \; t_{p+q}\cdots t_1,\nonumber
\end{align}
where we sum over all ordered subsets $f$ and all possible permutations $\tau \in S_{s}$. 

By construction, \eqref{expression with sigma (X)} and \eqref{A} are identical, which implies
\begin{align}\label{BB}
&\sigma\left(XZ_1\cdots Z_{p+q}\right) =\\
&\sum_{s=0}^{p+q} \sum\limits_{\substack{ f\\ \abs{f}=s}} \sum_{\tau}  \frac{(-1)^{\delta(\tau,f)}}{(s+1)!}  [[\cdots [X,Z_{\tau(f(s))}], \cdots],Z_{\tau(f(1))}]\sigma_\alpha(\partial_{z_{f(1)}} \cdots \partial_{z_{f(s)}} z_1 \cdots z_{p+q}) \nonumber
\end{align}

To write this in the proposed form, we associate with each $f$ the unique $K_f=(k_1, \ldots, k_{m+n}) \in \mN^{m|n}$ which satisfies
$$X_1^{k_1}\cdots X_{m+n}^{k_{m+n}}= Z_{f(1)}\cdots Z_{f(s)},$$
so in particular $\abs{K_f}=s$. This definition implies
\[
\xi_\alpha(\partial_{z_{f(1)}} \cdots \partial_{z_{f(s)}} z_1 \cdots z_{p+q}) = \frac{(\alpha-K_f)!}{\alpha!} \partial^{K_f} Y,
\]
which allows to rewrite equation \eqref{BB} as
\begin{align*}
\sigma\left(XY\right) =\sum_{s=0}^{p+q} \sum\limits_{\substack{ f\\ \abs{f}=s}} \sum_{\tau}  \frac{(-1)^{\delta(\tau,f)}(\alpha-K_f)!}{(|K_f|+1)!\alpha!}  [[\cdots [X,Z_{\tau(f(s))}], \cdots],Z_{\tau(f(1))}]\sigma( \partial^{K_f} Y).
\end{align*}
Introducing the symmetrisation map $\sigma$ then yields
\begin{align*}
\sigma\left(XY\right)&=\sum_{s=0}^{p+q} \sum\limits_{\substack{ f\\ \abs{f}=s}} \frac{(-1)^{|K_f|}(\alpha-K_f)!|K_f|!}{(|K_f|+1)!\alpha!} s^{K_f}_{\fg}(X)\;\sigma( \partial^{K_f} Y).
\end{align*}
It hence remains to interpret the summation in the right-hand side and compare to the one in equation \eqref{equation lemma 1}. Concretely we need to consider the map $q:f\mapsto K_f$. Firstly, this map implies that the summation in the above is not over all $K\in\mN^{m|n}$, but only over $K$ such that $K\le \alpha$. However, when that condition on $K$ is not satisfied we have $\partial^KY=0$. Secondly, the map is not injective. When $K\le \alpha$, the cardinality of $q^{-1}(K)$ is clearly $\alpha! /(K! (\alpha-K)!)$. Hence we obtain precisely equation \eqref{equation lemma 1}.
\end{proof}

\begin{lemma} \label{lemma Lie bracket of sigma}
Let $\mg$ be a  Lie superalgebra with basis $ X_i,$ $1\leq i \leq m+n$ and $Y$ be an element of~$S(\mg)$.
Then
\begin{align} \label{equation Lie bracket of sigma}
[X,\sigma(Y)]=\sigma\left(\sum_{i=1}^{m+n} [X,X_i] \partial^i  Y\right)=\sigma\left(\sum_{|K|=1}s_{\fg}^K(X)\partial^KY\right).
\end{align}
\end{lemma}
\begin{proof}
By linearity we can again assume $X$ to be homogeneous and $Y$ to be of the form $X_1^{\alpha_1}\cdots X_{m+n}^{\alpha_{m+n}}$. We will again write $Y$ as $Z_1\cdots Z_{p+q}$, where the $Z_i$ are defined in the same way as in the proof of Lemma~\ref{lemma 1}.

Assume $X$ to be odd. The case~$X$ even can be shown in a similar way. Starting from the left-hand side of \eqref{equation Lie bracket of sigma}, we get
\begin{align} \label{equation lemma Lie bracket sigma}
[X,\sigma(Y)] &= \frac{1}{(p+q)!} \sum_{\tau \in S_{p+q}} (-1)^{\abs{\tau}} [X,Z_{\tau(1)}Z_{\tau(2)} \cdots Z_{\tau(p+q)}] \nonumber \\
&= \frac{1}{(p+q)!}\sum_{\tau \in S_{p+q}} \sum_{i=1}^{p+q} (-1)^{\alpha(\tau,i)} Z_{\tau(1)} \cdots Z_{\tau(i-1)}[X,Z_{\tau(i)} ]Z_{\tau(i+1)} \cdots Z_{\tau(p+q)},
\end{align}
where\[
\alpha(\tau,i)=\sum_{j=1}^{i-1} [ \tau(j) > p] + \sum_{l=p+1}^{p+q-1} \sum_{j=l+1}^{p+q} [ \tau^{\minus 1}(l) > \tau^{\minus 1} (j)].
\]
We  can rewrite the right-hand side of \eqref{equation Lie bracket of sigma} using the notation $Z_1 \cdots \hat{Z}_k \cdots Z_{p+q}$ for the product of all $Z_i$ without the term $Z_k$ as 
\begin{align*}
& \sigma\left(\sum_{i=1}^{m+n} [X,X_i] \partial^i  Y\right) \\&= \sigma\left(\sum_{k=1}^{p} [X,Z_k] Z_1 \cdots \hat{Z}_k \cdots Z_{p+q} \right)+ \sigma\left(\sum_{k=p+1}^{p+q} (-1)^{k-1-p}[X,Z_k] Z_1 \cdots \hat{Z}_k \cdots Z_{p+q} \right) \\
&= \frac{1}{(p+q)!} \sum_{k=1}^p \sum_{\tau \in S_{p+q}} (-1)^{\beta(\tau,k)} \bar{Z}_{\tau(1)}^{(k)} \cdots \bar{Z}^{(k)}_{\tau(p+q)} + \frac{1}{(p+q)!} \sum_{k=p+1}^{p+q} \sum_{\tau \in S_{p+q}} (-1)^{\gamma(\tau,k)} \bar{Z}_{\tau(1)}^{(k)} \cdots \bar{Z}^{(k)}_{\tau(p+q)}, 
\end{align*}
where 
\begin{itemize}
\item $\bar{Z}^{(k)}_i = \begin{cases} Z_i \qquad \text{ for } i\not=k \\ [X,Z_i] \qquad \text{ for } i=k. \end{cases} $
\item $ \beta(\tau,k) = \sum_{l=p+1}^{p+q} [ \tau^{\minus 1}(l) < \tau^{\minus 1} (k) ] + \sum_{l=p+1}^{p+q-1}\sum_{r=l+1}^{p+q} [ \tau^{\minus 1}(l) > \tau^{\minus 1} (r) ]$ \item $ \gamma(\tau,k) =k-1-p + \sum_{l=p+1, \; l\not= k}^{p+q}  \sum_{r=l+1, \; r \not=k}^{p+q} [ \tau^{\minus 1}(l) > \tau^{\minus 1} (r) ].$
\end{itemize}
One can calculate that $\beta (\tau, k)= \gamma (\tau ,k) = \alpha (\tau,\tau^{\minus 1}(k))$. Rewriting \eqref{equation lemma Lie bracket sigma} as 
\begin{align*} 
[X,\sigma(Y)] 
&= \frac{1}{(p+q)!} \sum_{k=1}^{p+q}\sum_{\tau \in S_{p+q}} (-1)^{\alpha(\tau,\tau^{\minus 1}(k))} Z_{\tau(1)} \cdots [X,Z_k ] \cdots Z_{\tau(p+q)},
\end{align*}
concludes the proof.
 \end{proof}

\begin{lemma} \label{Lemma 3}
For any $K\in\mN^{m|n}$, we have 
\begin{align*}
\sum_{L < K} \frac{ B_{\abs{L}}K!}{(\abs{K-L}+1)(K-L)!L!}     s^L_{\mg}s^{K-L}_{\mg}(X) \; \partial^L \partial^{K-L} Y 
= - B_{\abs{K}} s^K_{\mg}(X) \partial^K (Y).
\end{align*}
\end{lemma}
\begin{proof}
For any $K,L\in \mN^{m|n}$ with $L<K$, we define $\gamma_{K,L}\in\mZ_2$ by 
\[
\partial^L \partial^{K-L}= (-1)^{\gamma_{K,L}} \partial^K.
\] 
We claim that for every $i < \abs{K}$, 
\begin{equation}\label{last}
\abs{K}! s^K_{\mg}(X) = \sum_{L < K, \abs{L}=i} \abs{L}!\abs{(K-L)}! \binom{K}{L}(-1)^{\gamma_{K,L}} s^L_{\mg}(s^{K-L}_\mg(X)).
\end{equation}
Indeed, we start from equation \eqref{defeqS} and consider one term in the expansion of the symmetrisation. This term corresponds to $|K|$ consecutive $\widetilde\ad$-operators acting on $X$. We fix the first $i$ operators from the left and now gather all other terms which start with this fixed sequence. This gives, up to an overall constant, the consecutive action of some $s_{\fg}^L$ with $|L|=|K|-i$, followed by the fixed $i$ operators. Now we also consider all terms in the expansion of $s^K_{\mg}(X)$ where the first $i$ of the $\widetilde\ad$-operators correspond to a permutation of the ones we considered earlier. Adding all these together gives a term $s^L_{\mg}(s^{K-L}_\mg(X))$, again up to multiplicative constant. All the terms in $s^K_{\mg}(X)$ that have not yet been considered can also be gathered in such
forms, for some different $L'\in\mN^{m,n}$ with $|L'|=|K|-i$. Keeping track of all constants and signs then yields \eqref{last}.

Now using the definition of the Bernoulli numbers \eqref{Bernoulli numbers} and equation \eqref{last}, we obtain 
\begin{align*}
- B_{\abs{K}} s^K_{\mg}(X) \partial^K (Y) &=
\sum_{i=0}^{\abs{K}-1} \frac{ B_i \abs{K}!}{(\abs{K}-i+1)!i!} s^K_{\mg}(X)\partial^K(Y) \\ &= \sum_{L<K} \frac{ B_{\abs{L}} }{(\abs{K-L}+1)! \abs{L}!}  \abs{L} ! \abs{K-L} !\binom{K}{L} s^L_{\mg}(x)s^{K-L}_\mg(X) \partial^L \partial^{K-L}(Y).
\end{align*}
which proves the lemma.
\end{proof}

\noindent
{\bf Acknowledgment.}
The authors thank Hendrik De Bie for useful discussions.
SB is a PhD Fellow of the Research Foundation - Flanders (FWO). KC is supported by the Research Foundation - Flanders (FWO) and by Australian Research Council Discover-Project Grant DP140103239.

\vspace{-2.5mm}

\noindent
SB: Department of Mathematical Analysis, Faculty of Engineering and Architecture, Ghent University, Krijgslaan 281, 9000 Gent, Belgium;
E-mail: {\tt Sigiswald.Barbier@UGent.be} 

\vspace{-1mm}

\noindent
KC: School of Mathematics and Statistics, University of Sydney, NSW 2006, Australia;
E-mail: {\tt kevin.coulembier@sydney.edu.au}

\date{}

\end{document}